\documentclass[11pt]{amsart}

\usepackage{amsmath, amssymb, amscd, mathrsfs, url, pinlabel,setspace,hyperref,verbatim}
\usepackage[margin=1.25in,marginparwidth=1in,centering,letterpaper,dvips]{geometry}
\usepackage{color,dcpic,latexsym,graphicx,epstopdf,comment}
\usepackage[all]{xy}
\usepackage{tikz,pgfplots}
\usepackage{tikz-cd,circuitikz}
\usepackage{enumitem,upquote,tabularx,textcomp,setspace,color}
\usepackage{makecell}
\usepackage{caption} 
\xyoption{rotate}
\DeclareMathAlphabet{\mathpzc}{OT1}{pzc}{m}{it}
\SetMathAlphabet{\mathpzc}{bold}{OT1}{pzc}{b}{it}

\newtheorem*{rep@theorem}{\rep@title}
\newcommand{\newreptheorem}[2]{%
\newenvironment{rep#1}[1]{%
 \def\rep@title{#2 \ref{##1}}%
 \begin{rep@theorem}}%
 {\end{rep@theorem}}}
\makeatother

\newtheorem {theorem}{Theorem}
\newreptheorem{theorem}{Theorem}
\newtheorem {lemma}[theorem]{Lemma}
\newtheorem {proposition}[theorem]{Proposition}
\newtheorem {corollary}[theorem]{Corollary}

\numberwithin{equation}{section}
\numberwithin{theorem}{section}

\theoremstyle{definition}
\newtheorem{definition}[theorem]{Definition}
\newtheorem{construction}[theorem]{Construction}
\newtheorem{data}[theorem]{Data}
\newtheorem{notation}[theorem]{Notation}

\newtheorem{remark}[theorem]{Remark}

\newtheorem{example}[theorem]{Example}
\newtheorem*{ack}{Acknowledgement}
\newtheorem*{org}{Organization}

\newlist{pcases}{enumerate}{1}
\setlist[pcases]{
  label=\bf{Case~\arabic*:}\protect\thiscase.~,
  ref=\arabic*,
  align=left,
  labelsep=0pt,
  leftmargin=0pt,
  labelwidth=0pt,
  parsep=0pt
}
\newcommand{\case}[1][]{%
  \if\relax\detokenize{#1}\relax
    \def\thiscase{}%
  \else
    \def\thiscase{~#1}%
  \fi
  \item
}

\newcommand{\bu}{\bullet}

\newcommand{\Z}{\mathbb{Z}}

\newcommand{\R}{\mathbb{R}}
\newcommand{\C}{\mathbb{C}}

\newcommand{\CP}{\mathbb{CP}}

\newcommand{\Q}{\mathbb{Q}}

\newcommand{\spinc}{\operatorname{Spin}^c}

\newcommand{\topo}{\mathcal{T}}

\DeclareMathOperator{\id}{id}

\DeclareMathOperator{\colim}{colim}

\newcommand{\bF}{\mathbb{F}}

\newcommand{\cB}{\mathcal{B}}

\newcommand{\cD}{\mathcal{D}}

\newcommand{\sR}{\mathscr{R}}

\newcommand{\cR}{\mathcal{R}}

\newcommand{\bI}{\mathbb{I}}

\newcommand{\fs}{\mathfrak{s}}
\newcommand{\ft}{\mathfrak{t}}

\newcommand{\ov}[1]{{\overline{#1}}}

\newcommand{\wti}[1]{{\widetilde{#1}}}

\DeclareFontFamily{U}{mathx}{\hyphenchar\font45}
\DeclareFontShape{U}{mathx}{m}{n}{
      <5> <6> <7> <8> <9> <10>
      <10.95> <12> <14.4> <17.28> <20.74> <24.88>
      mathx10
      }{}
\DeclareSymbolFont{mathx}{U}{mathx}{m}{n}
\DeclareFontSubstitution{U}{mathx}{m}{n}
\DeclareMathAccent{\widecheck}{0}{mathx}{"71}

\newcommand{\Dc}{D^{\mathrm{c}}}
\newcommand{\Dcc}{D^{\mathrm{cc}}}

\newcommand{\pt}{\mathrm{pt}}

\newcommand{\bfB}{\mathbf{B}}

\usetikzlibrary{calc,intersections}
\tikzset{every picture/.style=thick}
\tikzset{link/.style = { white, double = black, line width = 1.75pt, double distance = 1.25pt, looseness=1.75 }}
\tikzset{crossing/.style = {draw, circle, dotted, minimum size=0.5cm, inner sep=0, outer sep=0}}
\pgfplotsset{compat=1.12}

\usepackage{extarrows}
\usepackage{mathabx}
\usepackage{amsbsy}
\usepackage{appendix}
\usepackage{float}
\usepackage{subfigure}
\usepackage[numbers,sort&compress]{natbib}
\usepackage{amsthm}
\usepackage{geometry}
\usepackage{fancyhdr}
\usepackage{overpic}
\usepackage{mathrsfs}
\usepackage{colonequals}
\usepackage{microtype}
\usepackage{diagbox}
\usepackage{tabularx}

\newcommand{\bpf}{\begin{proof}}
\newcommand{\epf}{\end{proof}}
\newcommand{\bthm}{\begin{theorem}}
\newcommand{\ethm}{\end{theorem}}
\newcommand{\bprop}{\begin{proposition}}
\newcommand{\eprop}{\end{proposition}}
\newcommand{\bcor}{\begin{corollary}}
\newcommand{\ecor}{\end{corollary}}
\newcommand{\blem}{\begin{lemma}}
\newcommand{\elem}{\end{lemma}}
\newcommand{\bdefn}{\begin{definition}}
\newcommand{\edefn}{\end{definition}}
\newcommand{\bcons}{\begin{construction}}
\newcommand{\econs}{\end{construction}}
\newcommand{\bdata}{\begin{data}}
\newcommand{\edata}{\end{data}}
\newcommand{\bexmp}{\begin{example}}
\newcommand{\eexmp}{\end{example}}
\newcommand{\brem}{\begin{remark}}
\newcommand{\erem}{\end{remark}}
\newcommand{\bnot}{\begin{notation}}
\newcommand{\enot}{\end{notation}}
\newcommand{\benu}{\begin{enumerate}}
\newcommand{\benum}{\begin{enumerate}[leftmargin=*]}
\newcommand{\eenu}{\end{enumerate}}
\newcommand{\beq}{\begin{equation}}
\newcommand{\eeq}{\end{equation}}

\newcommand{\al}{\alpha}
\newcommand{\be}{\beta}
\newcommand{\ga}{\gamma}
\newcommand{\Ga}{\Gamma}

\newcommand{\pa}{\partial}
\newcommand{\ot}{\otimes}
\newcommand{\op}{\oplus}

\newcommand{\aand}{~\mathrm{and}~}

\definecolor{lygreen}{HTML}{016646}
\title{Monopole triangle over integers}

\author{Haochen Qiu}
\address{Department of Pure Mathematics and Mathematical Statistics\\University of Cambridge}
\email{hq238@cam.ac.uk}

\author{Fan Ye}
\address{School of Mathematical Sciences\\Peking University}
\email{flyye@math.pku.edu.cn}

\makeatletter

\begin{document}

\begin{abstract}
We prove the surgery exact triangle for monopole (Seiberg--Witten) Floer homology over integer coefficients, extending the work of Kronheimer--Mrowka--Ozsv\'{a}th--Szab\'{o} over $\mathbb{Z}/2$, Lin--Ruberman--Saveliev over $\mathbb{Q}$, and Freeman over $\mathbb{Z}[\sqrt{-1}]$. Our proof is based on a modification of Kronheimer--Mrowka's local system on monopole Floer homology and an adaptation of Freeman's computation. As a standard application, following Bloom and Scaduto, we obtain a spectral sequence $\widetilde{Kh}_{\mathrm{odd}}(L)\Rightarrow \widetilde{HM}_\bullet(-\Sigma_2(L))$ over integer coefficients for an oriented link $L\subset S^3$, thereby solving Ozsv\'{a}th--Rasmussen--Szab\'{o}'s conjecture.

\end{abstract}
\maketitle
\section{Introduction}
The surgery exact triangle is a powerful tool to study Floer homology of $3$-manifolds via Dehn surgery on a knot. It was initially introduced by Floer \cite{floer1988instanton, floer1990knot} in instanton theory and later reviewed in \cite{donaldson95triangle,scaduto2015instantons}. Subsequently, analogous triangles were developed in monopole, Heegaard Floer, singular instanton, and variants of those Floer theories \cite{kronheimer2007monopoles, kronheimer2007monopolesandlens, Ozsvath2004c, ozsvath2005double,kronheimer2011khovanov,Lin17triangle,HHSZ25triangle,bhat2023newtriangle,jiakai2023triangle}. The three maps in the surgery exact triangle are typically the cobordism maps associated to the surgery cobordisms. In this setting, iterating the triangle for a link yields a spectral sequence whose second page is (a variant of) the Khovanov homology of the link \cite{Bloom2011link,scaduto2015instantons,KM2021Barnatan,Lin19involutive,ATZ23ss,jiakai2023triangle}.

The triangle in instanton theory was constructed over integer coefficients, whereas those in monopole and Heegaard Floer theories are usually only over $\Z/2$ coefficients. The choice of the coefficients is essential, as the composition of two consecutive surgery cobordism maps in the surgery exact triangle only has a factor of $2$ and does not vanish in general; see \cite[\S 8]{KP2021bordered} and \cite[Remark 5.3]{Freeman2021triangle}.

An approach to resolving this issue is to introduce some ad hoc sign arrangement on each spin$^c$ component of the cobordism maps so that the composition of the modified cobordism maps vanishes; see Ozsv\'{a}th--Szab\'{o} \cite[Proposition 9.6]{Ozsvath2004c} for the Heegaard Floer triangle over $\Z$ and Lin--Ruberman--Saveliev \cite[Proposition 4.4]{LRS2023triangle} for the monopole triangle over $\Q$. However, such ad hoc arrangements make it hard to generalize the triangle to the link spectral sequence.

Another approach is to introduce some local coefficients (also called twisted coefficients) so that the signs appear naturally. Recently, Abouzaid--Manolescu \cite{AM2025sign} proposed a more canonical sign arrangement in the Heegaard Floer theory and proved the naturality and the surgery exact triangle over $\Z$. For the triangle, they used some homological trick to get rid of the twisted coefficients of all three Floer homology groups but keep the cobordism maps twisted. A similar trick was used by Baldwin--Sivek \cite[\S 2.2]{baldwin2020concordance} in instanton theory. One may expect that Abouzaid--Manolescu's triangle can be iterated for a link, but some extra care may be needed for the integral lifts of the cobordism maps on the first page of the spectral sequence to obtain the odd Khovanov homology. We will not pursue this direction in this paper. The advantage of monopole Floer theory is that the cobordism map over $\Z$ is already well-defined in Kronheimer--Mrowka's book \cite{kronheimer2007monopoles}.

In monopole theory, Freeman's PhD thesis \cite{Freeman2021triangle} introduced a local system with fiber $\Z[\sqrt{-1}]$ to prove the surgery exact triangle over the associated local coefficients. This local system is somewhat nonstandard and the related computations become complicated.

In this paper, we transfer Kronheimer--Mrowka's construction of the local system in singular instanton theory in \cite{kronheimer2011knot,KM2019deformation,KM2021Barnatan,KM2025relation} to the monopole theory and adapt Freeman's computation to obtain the monopole surgery exact triangle over $\Z$. We use a similar homological trick so that none of the three Floer homology groups carry local coefficients and only the cobordism maps are twisted. 

This construction can be easily iterated for a link and induces a spectral sequence from the reduced odd Khovanov homology $\widetilde{Kh}_{\mathrm{odd}}(L)$ of a link $L$ to the tilde version of monopole Floer homology $\widetilde{HM}_\bu(-\Sigma_2(L))$ of the double branched cover $\Sigma_2(L)$ with the opposite orientation. 

An advantage of the Floer homology over $\Z$ is in the context of contact elements: a vector space over $\Z/2$ only has finitely many elements, whereas a $\Z$-module can have infinitely many elements realized by contact elements; see \cite[Theorem 2 and Corollary 2]{Massot12contact} for the example of $T^3$ in Heegaard Floer homology. On the other hand, see also \cite{JM2008HF} for the existence of torsion elements in Heegaard Floer homology of a closed oriented surface times a circle.

Meanwhile, since the (ordinary) homology with local coefficients is related to the homology of some covering space \cite[\S 3.H]{hatcher2002AT}, our work may suggest a better way to construct the surgery exact triangle for Seiberg--Witten Floer spectra mentioned in \cite[\S 3.6]{DSS23spectra}.

Now we start to state the main theorems of the paper. For simplicity, let $HM$ denote any fixed version of monopole Floer homology \begin{equation*}\label{eq: HM}
    \widehat{HM}_\bu \text{ (from)}, ~\widecheck{HM}_\bu\text{ (to)},~ \overline{HM}_\bu\text{ (bar)},\text{ and } \widetilde{HM}_\bu\text{ (tilde)},
\end{equation*}where the first three are defined by Kronheimer--Mrowka \cite{kronheimer2007monopoles} and the last one is defined by Bloom \cite{Bloom2011link} as the mapping cone of the $U$-action. Through the celebrated result ``$HF=HM=ECH$" \cite{kutluhan2010hf,taubes2010ech1,colin2011equivalence}, the four versions of $HM$ are isomorphic to the corresponding versions of Heegaard Floer homology \begin{equation*}\label{eq: HF} \mathbf{HF}^- \text{ (minus)}, ~\mathbf{HF}^+\text{ (plus)},~ \mathbf{HF}^\infty\text{ (infinity)},\text{ and } \widehat{\mathbf{HF}}\text{ (hat)},
\end{equation*}where the bold symbols denote the homology of the completion of the original Heegaard Floer chain complex from $\Z[U]$ to $\Z[[U]]$, which equals the original homology for the plus and hat versions.

\bdefn\label{defn: surgery triad}
A \emph{surgery triad} $(Y_0,Y_1,Y_2)$ is a cyclically ordered tuple of closed $3$-manifolds $Y_i$ for $i\in\Z/3$ obtained from a $3$-manifold $M$ with torus boundary by Dehn filling along an oriented simple closed curve $\ga_i\subset \partial M\cong T^2$ such that the algebraic intersection numbers satisfy\[\ga_0\cdot \ga_1=\ga_1\cdot \ga_2=\ga_2\cdot \ga_0=-1.\]Equivalently, for some fixed $i\in\Z/3$, let $K_i$ be the core knot $S^1\times\{0\}$ of the Dehn filling solid torus $S^1\times D^2$ for $\ga_i$ with the framing induced by $\ga_{i+1}$. Then a surgery triad consists of Dehn surgery manifolds obtained from $Y_i$ with slopes $(\infty,0,1)$. The neighborhoods of the core knots are called the \emph{surgery region}. There is a cobordism $W_i:Y_i\to Y_{i+1}$ obtained from $Y_i\times I$ by attaching $2$-handle along $\ga_i\times \{1\}$ with framing $\ga_{i+1}$, called the \emph{surgery cobordism}.
\edefn
The following is the first main theorem of the paper.
\bthm\label{thm: main triangle}
Let $(Y_0,Y_1,Y_2)$ be a surgery triad and let $W_i:Y_i\to Y_{i+1}$ for $i\in\Z/3$ be the corresponding surgery cobordism. Then there exists an exact triangle (over integer coefficients)
\[	\xymatrix{
	HM(Y_0)\ar[rr]^{HM(W_0;\Ga_{\nu_0})}&& HM(Y_1)\ar[dl]^{\quad HM(W_1;\Ga_{\nu_1})}\\
	&HM(Y_2)\ar[lu]^{HM(W_2;\Ga_{\nu_2})\quad}&
	}\] where $HM(W_i;\Ga_{\nu_i})$ is the cobordism map for some local system associated to some smooth relative $2$-chain $\nu_i\subset W_i$.
\ethm
\brem
Recall that the total sign of a cobordism map depends on the choice of the homology orientation. Note that the kernel and the image of a map are independent of the sign, so is the exactness. Hence we do not need to specify the homology orientations of the cobordism maps in Theorem \ref{thm: main triangle}.
\erem
\brem
 Indeed, one of $\nu_i$ in Theorem \ref{thm: main triangle} is empty and the corresponding cobordism map is the usual one over $\Z$ without local system. For the other two $\nu_i$, one is the union of a core disk in the corresponding cobordism and a $2$-chain in the source, the other is the union of a cocore disk and the same $2$-chain with opposite sign in the target. See Corollary \ref{cor: triangle 2}. We will show in \S \ref{sec: Triangle over integer coefficients} that the cobordism maps in Theorem \ref{thm: main triangle} recover those in \cite[\S 4]{LRS2023triangle} over $\Q$; see Proposition \ref{prop: recover LRS}.
\erem
The exact triangle in Theorem \ref{thm: main triangle} can be iterated to obtain the following spectral sequence.
\bthm\label{thm: main ss}
Let $L$ be a framed link in a closed oriented $3$-manifold $Y$ with ordered $l$ components. Given $v\in\{0,1\}^l$, let $Y_{v}=Y_v(L)$ be obtained from $Y$ by Dehn surgery on $L$ via the slopes in $v$. Then there exists a spectral sequence depending only on $(Y,L)$\[E^1=\bigoplus_{v\in \{0,1\}^l}HM(Y_v)	\Longrightarrow HM(Y;\Ga_{L}),\]where $\Ga_{L}$ is some local system associated to $L$ and $HM(Y;\Ga_{L})$ is a $\Z$-module (or $\Z[[U]]$-module for some version of $HM$) whose isomorphism class only depends on the homology class $[L]\in H_1(Y;\Z/2)$. Moreover, we have\[d^1=\sum_{\substack{|v-w|=1\\v<w}}\pm HM(W_{vw}),\]where $|\cdot|$ is the $L^1$-norm, i.e.\ the number of $1$s in the entry, $<$ is the partial order induced by $0<1$, $W_{vw}=W_{vw}(L)$ is the surgery cobordism with source $Y_v$ and target $Y_w$, and the sign is induced by suitable homology orientation and the entries of $v$ and $w$. Note that there is no local system in $d^1$. Moreover, the spectral sequence inherits a $\Z\op\Z/2$-grading so that the grading shift of the differential $d^r$ in the $r$-th page is $(r,1)$.
\ethm
Following Bloom \cite[Theorem 1.3]{Bloom2011link} and Scaduto \cite[Theorem 1.1]{scaduto2015instantons}, the second page in Theorem \ref{thm: main ss} is related to the odd Khovanov homology.
\bthm\label{thm: main kh}
Let $L\subset S^3$ be an oriented link and let $\cD$ be a planar diagram of $L$. Then there exists a spectral sequence depending only on $(L,\cD)$\[E^2=\widetilde{Kh}_{\mathrm{odd}}(L)=\widetilde{Kh}_{\mathrm{odd}}(\cD)	\Longrightarrow \widetilde{HM}_\bu(-\Sigma_2(L)).\]Moreover, the spectral sequence inherits a $\Z\op\Z/2$-grading so that the grading shift of the differential $d^r$ in the $r$-th page is $(r,1)$, and the grading on $\widetilde{Kh}_{\mathrm{odd}}$ is given by \[\delta-\frac{\sigma(L)+\nu(L)}{2}=-t+\frac{q-\sigma(L)+\nu(L)}{2}\mod 2\]for the signature $\sigma$, the nullity $\nu$, the quantum grading $q$, and the homological grading $t$.
\ethm
\brem\label{rem: solve conj}
Together with the isomorphism $\widetilde{HM}_\bu(-\Sigma_2(L))\cong \widehat{HF}(-\Sigma_2(L))$, Theorem \ref{thm: main kh} solves Ozsv\'{a}th--Rasmussen--Szab\'{o}'s conjecture \cite[Conjecture 1.1]{ORS2013oddkh}. From \cite[Proposition 1.8]{ORS2013oddkh}, the unreduced odd Khovanov homology is the direct sum of two copies of the reduced one (with grading shifts), which is a different phenomenon from the (even) Khovanov homology. Meanwhile, the reduced odd Khovanov homology can contain substantial $\Z$-torsion so the spectral sequence over $\Z$ is significantly stronger than that over $\Z/2$. For example, from \cite[\S 5]{ORS2013oddkh}, we know that $\widetilde{Kh}_{\mathrm{odd}}(9_{42};\Q)$ is thin, while $\widetilde{Kh}_{\mathrm{odd}}(9_{42};\Z/2)$ is not thin. See also Shumakovitch \cite[\S 4.2]{Shu11oddkh} for $3$-torsion in the reduced odd Khovanov homology of $9_{46}$ and $10_{140}$, which are $(3,3,-3)$- and $(3,4,-3)$-pretzel knots. Note that they are both thin when $3$ is invertible. Shumakovitch also mentioned that $(n, n, -n)$- and $(n, n + 1, -n)$-pretzel links have torsion of order $n$ for positive integer $n\le 6$.
\erem
\brem

We conjecture that the spectral sequence in Theorem \ref{thm: main kh} is also independent of the planar diagram $\cD$, similar to the results by Baldwin \cite{Baldwin2011ss} and Bloom \cite[\S 9.2]{Bloom2011link} for Heegaard Floer and monopole Floer homologies over $\Z/2$, respectively. If this is true, then we would obtain that these spectral sequences only depend on the mutation equivalence class of $L$ as in \cite[Theorem 1.4]{Bloom2011link}. Similar to the work of Baldwin--Hedden--Lobb \cite{BHL2019ss}, one can further ask the functoriality of the spectral sequence under link cobordisms.
\erem
From Remark \ref{rem: solve conj} and the similar spectral sequence for framed instanton homology $I^\sharp$ by Scaduto \cite[Theorem 1.1]{scaduto2015instantons}, we have the following corollary, which includes the knots $9_{42}$, $9_{46}$, and $10_{140}$ as motivating examples.
\bcor
Given a link $L\subset S^3$ and a coefficient field $\bF$, if $\widetilde{Kh}_{\mathrm{odd}}(L;\bF)$ is totally supported in even $\delta$-gradings or odd $\delta$-gradings, then the spectral sequence in Theorem \ref{thm: main kh} collapses and\[\widetilde{Kh}_{\mathrm{odd}}(L;\bF)\cong \widehat{HF}(-\Sigma_2(L);\bF)\cong \widetilde{HM}_\bu(-\Sigma_2(L);\bF)\cong I^\sharp(-\Sigma_2(L);\bF).\]
\ecor
\begin{org}
In \S \ref{sec: Local system}, we construct the local system $\Ga_\eta$ over $\Z[u, u^{-1}]$ for a formal variable $u$ (distinct from the $U$-action) and prove some properties when $u=-1$. We will only use the case $u=-1$ in later sections and use the local system sets $\eta$ and $\nu$ in $\Ga_{\eta}$ and $\Ga_{\nu}$ to denote this specific local system. We also use the terms ``local system'' and ``local coefficients'' interchangeably.

In \S \ref{sec: Triangle over local system}, we adapt Freeman's computation \cite[\S 7]{Freeman2021triangle} to obtain the monopole surgery exact triangle over our local system. For the sake of completeness, we add many details to Freeman's original proof.

In \S \ref{sec: Triangle over integer coefficients}, we use a homological trick to transfer the triangle over the local system to one without local system on all three Floer homology groups (Theorem \ref{thm: main triangle}). We also show that the cobordism maps over the local system coincide with the ones based on ad hoc sign arrangements used by Lin--Ruberman--Saveliev.

In \S \ref{sec: Spectral sequence}, we follow the work of Bloom \cite{Bloom2011link} and Scaduto \cite{scaduto2015instantons} to iterate the construction for a link and prove Theorems \ref{thm: main ss} and \ref{thm: main kh}.

In \S \ref{sec: Further direction on sutured monopole theory}, we sketch how our local system with $u=-1$ interacts with the sutured monopole theory in \cite{kronheimer2010knots}, indicating some further directions.
\end{org}
\begin{ack}
The first author thanks Sungkyung Kang for helpful discussions. The second author thanks Mohammed Abouzaid for helpful discussions that motivated this work. The authors also thank Ciprian Manolescu for comments on the draft of the paper.

\end{ack}

\section{Local system}\label{sec: Local system}
In this section, we construct a local system over $\Z[u,u^{-1}]$ for monopole Floer homology. This local system differs from those in \cite[\S 22.6 Examples, the end of \S 23.3]{kronheimer2007monopoles} and \cite[\S 2.2]{kronheimer2010knots} but still follows the framework of \cite[Theorem 23.3.4]{kronheimer2007monopoles}. The construction is motivated by the local system in singular instanton theory \cite{kronheimer2011knot,KM2019deformation,KM2021Barnatan,KM2025relation}. 

We will use the notation from \cite{kronheimer2007monopoles} freely. We write $i=\sqrt{-1}\in\C$ in the formula about the integral of the curvature. Throughout this section, let $HM$ denote any fixed version of\[\widehat{HM}_\bu, ~\widecheck{HM}_\bu,~ \overline{HM}_\bu,\text{ and } \widetilde{HM}_\bu\]and let $CM$ be the corresponding monopole Floer chain complex.

\subsection{Kronheimer--Mrowka's construction}\label{subsec: Kronheimer--Mrowka's construction}
We first recall the local system in \cite[\S 22.6 and \S 23.3]{kronheimer2007monopoles} and \cite[\S 2.2]{kronheimer2010knots}, denoted by $\Ga'_{\eta}$. 

Suppose $Y$ is a (smooth, oriented, connected) closed $3$-manifold and $\fs$ is a spin$^c$ structure on $Y$. Suppose $\cR$ is a commutative ring with identity and an exponential map $\exp:\R\to \cR^\times$ and write $t=\exp(1)\in\cR$. Given a smooth $1$-cycle $\eta$ in $Y$ with real coefficients, we define the local system $\Ga'_\eta$ on the blown-up monopole configuration space $\cB^\sigma(Y,\fs)$ as follows. The fiber on $\bfB=[B, r,\psi] \in \cB^\sigma(Y,\fs)$ is $\cR$, where $B$ is a spin$^c$ connection, $r\in [0,\infty)$, and $\psi$ is a unit section of the spinor bundle. For a homotopy class of paths $z:[0,1]\to \cB^\sigma(Y,\fs)$, we obtain a gauge-equivalence class of $4$-dimensional connection $[A_z]$ on $[0,1]\times Y$ and assign the map\[\Ga'_\eta(z):\Ga'_\eta (z(0))\to \Ga'_\eta (z(1))\]by the multiplication of $t^{r(z)}=\exp(r_\eta(z))$ for \begin{equation}\label{eq: r(z)}
    r_\eta(z)=\frac{i}{2\pi}\int_{[0,1]\times\eta}F_{A_z^t}.
\end{equation}Here $F_{A_z^t}$ is the curvature of the connection $A_z^t$ in the associated line bundle and $r_\eta(z)$ does not depend on the choice of the representative $A_z$. This local system $\Ga'_\eta$ induces a construction of the monopole Floer chain complex \cite[\S 22.6]{kronheimer2010knots}
\[
CM(Y,\fs;\Gamma'_\eta) = \bigoplus_{\bfB \in \mathfrak{C}} \Z \Lambda_\eta(\bfB) \otimes \Gamma (\bfB),
\]where $\mathfrak{C}$ is the set of generators. The differential is obtained by
\begin{equation}\label{equ:chain-level-differential}
\partial = \sum_{\bfB_0,\bfB_1 \in \mathfrak{C} }\sum_z\sum_{[\gamma ]\in M_z( \bfB_0, \bfB_1)} \epsilon[\gamma] \otimes \Gamma_\eta(z),
\end{equation}where $M_z$ is the corresponding monopole Moduli space. Note that the chain complex of the tilde version $\widetilde{CM}_\bu$ was defined in \cite[\S 8]{Bloom2011link} as the mapping cone of the $U$-map in $\widecheck{CM}_\bu$. The homology is denoted by $HM(Y,\fs;\Ga'_\eta)$. We can extend the local system to \[\cB^\sigma(Y)=\coprod_{\fs\in\spinc(Y)}\cB^\sigma(Y, \fs)\] and obtain\[HM(Y;\Ga'_{\eta})=\bigoplus_{\fs\in\spinc(Y)}HM(Y,\fs;\Ga'_\eta).\]

Let $W:Y_0\to Y_1$ be a cobordism between two closed $3$-manifolds and let $\nu$ be a smooth relative $2$-chain $\nu$ with $\partial \nu=\eta_1-\eta_0$ for $\eta_i$ on $Y_i$ with $i=0,1$. We will call $(W,\nu)$ a cobordism from $(Y_0,\eta_0)$ to $(Y_1,\eta_1)$. Then $\nu$ induces a $W$-morphism
\[
\Gamma'_\nu: \Gamma'_{\eta_0} \to \Gamma'_{\eta_1}
\]
as follows. Let $\gamma = (A,s,\phi)$ be a representative of $[\gamma] \in \cB^\sigma(W)$ that connects $\bfB_0=[B_0,r_0,\psi_0] \in \cB^\sigma(Y_0)$ to $\bfB_1=[B_1,r_1,\psi_1] \in \cB^\sigma(Y_1)$. Then define
\[
\Gamma'_\nu([\gamma]): \Gamma'_{\eta_0}(\bfB_0) \to \Gamma'_{\eta_1}(\bfB_1)
\]
by the multiplication of
\begin{equation}\label{eq: W-morphism}
    \exp\big(\frac{i}{2\pi}\int_\nu F_{A^t}\big).
\end{equation}
This $W$-morphism induces a chain map\[CM(W;\Ga'_\nu):CM(Y_0;\Ga'_{\eta_0})\to CM(Y_1;\Ga'_{\eta_1})\]that depends on a choice of homology orientation of $W$, which then induces the monopole cobordism map \[HM(W;\Ga'_\nu):HM(Y_0;\Ga'_{\eta_0})\to HM(Y_1;\Ga'_{\eta_1}).\]

By Stokes' theorem, we know $\Ga'_\nu=\Ga'_{\nu+\partial \theta}$ for any smooth $3$-chain $\theta$. Moreover, if two $1$-cycles $\eta$ and $\eta'$ on $Y$ satisfy $\eta'-\eta=\partial \sigma$ for a smooth $2$-chain $\sigma$ on $Y$, then $HM(Y\times I;\Ga'_\sigma)$ induces an isomorphism (up to sign by the choice of homology orientation) from $HM(Y;\Ga'_{\eta})$ to $HM(Y;\Ga'_{\eta'})$. The difference of the isomorphisms for two choices $\sigma$ and $\sigma'$ is the multiplication by\[t^{\langle c_1(\fs),[\sigma-\sigma']\rangle}\]on the summand corresponding to $\fs\in\spinc(Y\times I)$. In particular, the isomorphism class of $HM(Y;\Ga'_{\eta})$ only depends on the homology class of $\eta$.

\subsection{Our construction}\label{subsec: Our construction}

We adopt the notation in \S \ref{subsec: Kronheimer--Mrowka's construction}, but take all smooth cycles and chains $\eta,\nu,\theta,\sigma$ over $\Z$ rather than $\R$. Hence we may regard them as unions of submanifolds (that may have intersections). Our local system $\Ga_\eta$ on $\cB^\sigma(Y,\fs)$ is obtained by a function\begin{equation*}\label{eq: function f}
    f:\cB^\sigma(Y,\fs)\to \R/\Z=S^1
\end{equation*}and pull back the standard local system on $S^1$ with fiber $u^{\lambda}R$ for $R=\Z[u^{\pm 1}]$ at $\lambda\in\R/\Z$.

The function $f$ is obtained by the holonomy of the connection around $\eta$, mentioned in \cite[\S 6.1]{Freeman2021triangle}. More precisely, let $\tau_{\det}(\eta)$ be a unit section of the determinant line bundle of the spin$^c$ bundle over $\eta$. For $\bfB=[B,r,\psi]\in\cB(Y)$, we define\[f(\bfB)=f_{\tau_{\det}(\eta)}(\bfB)=\frac{1}{2}\frac{i}{2\pi}\int_{\tau_{\det}(\eta)}B^t,\]where $B^t$ is the induced connection in the determinant line bundle, and the integral on $\tau_{\det}(\eta)$ denotes the integration of the pull-back of the connection. We write the factor $1/2$ separately because it is the crucial ingredient of our construction. See Remark \ref{rem: u=t^-2} below. 

The function $f$ is well-defined, because the spin$^c$ connections $B$ and $g(B)$ in two different representatives of $\bfB$ for a gauge action $g: Y\to S^1$ differ by a multiple of $g^{-1}dg\ot \id_S$ for the spinor bundle $S$ by \cite[Equation (4.5)]{kronheimer2007monopoles}. Hence $B^t$ and $g(B)^t$ differ by $2g^{-1}dg$ because $S$ is a complex bundle of rank two. Thus, we have\begin{equation}\label{eq: 2Z}
    \frac{i}{2\pi}\int_{\tau_{\det}(\eta)}(B^t-g(B)^t)=\frac{i}{2\pi}\int_{\tau_{\det}(\eta)}4\pi i[g]=-2\langle [g],[\eta]\rangle\in 2\Z.
\end{equation}It is necessary to use $\tau_{\det}(\eta)$ instead of $\eta$ itself because the integrals for different choices of $\tau_{\det}(\eta)$ differ by an integer. Hence (only) the parity of $\tau_{\det}(\eta)$ matters.

From the construction of the local system on $S^1$, the fiber of $\Ga_\eta$ at $\bfB$ is \[u^{f(\bfB)}R\]and the homotopy class of paths $z:[0,1]\to \cB^\sigma(Y,\fs)$ corresponds to the multiplication by \[u^{\wti{f}(z(1))-\wti{f}(z(0))}\]for any lift $\wti{f}:\cB^\sigma(Y,\fs)\to \R$ of $f$. By Stokes' theorem and the fact that $F^t=dB^t$ for the determinant line bundle, we obtain that\[\wti{f}(z(1))-\wti{f}(z(0))=\frac{1}{2}\frac{i}{2\pi}\int_{[0,1]\times\tau_{\det}(\eta)}F_{A_z}^t.\]Note that the right-hand-side does not depend on the choice of $\tau_{\det}(\eta)$ and is equal to $r_\eta(z)/2$ for $r_\eta(z)$ from \eqref{eq: r(z)}.

\brem\label{rem: u=t^-2}
Roughly, compared to Kronheimer--Mrowka's construction, we have $u=t^2$, where $2$ is from the factor $1/2$ (see \cite[\S 2.2 Remarks]{KM2021Barnatan}), though our choice of the fiber is different from theirs. In this spirit, Freeman \cite{Freeman2021triangle} considered the case $t=i$ and only obtained a triangle over $\Z[i]$, while we can set $u=-1$ and obtain a triangle over $\Z$.
\erem

The construction of the $W$-morphism is more subtle, as the definition in \eqref{eq: W-morphism} with factor $1/2$ does not work directly and we need some even result from the spin$^c$ condition similar to \eqref{eq: 2Z}. We first introduce the following definitions.

\bdefn\label{defn: framing}
Given a spin$^c$ structure $\fs$, we write $P_{\det}\fs$ for the corresponding determinant (complex) line bundle. Given an integral smooth $1$-cycle $\eta$ in $Y$, a \emph{framing} $\tau^{\fs}(\eta)$ of $\eta$ for a spin$^c$ structure $\fs$ on $Y$ is a tuple
    \[
    ( \tau_{\det}^\fs(\eta), \tau_N^\fs(\eta), \tau_T^\fs(\eta))
    \]
    where 
    \begin{enumerate}
        \item $\tau_{\det}^\fs(\eta)$ is a unit section of the determinant line bundle $P_{\det}\fs$ over $\eta$,
        \item $\tau_N^\fs(\eta)$ is a framing of the normal bundle of $\eta$ in $Y$, and
        \item $\tau_T^\fs(\eta)$ is a framing of the tangent bundle of $\eta$ as the boundary of $\eta \times [0,\epsilon) \subset Y \times [0,\epsilon)$, such that $(\tau_N^\fs \times \tau_T^\fs)(\eta) := \tau_N^\fs(\eta) \times \tau_T^\fs(\eta)$ is compatible with the orientation of $Y \times [0,\epsilon)$.
    \end{enumerate} 
    We call $\tau^\fs(\eta)$ an \emph{admissible framing} if for a $4$-dimensional spin$^c$ filling $(X,\fs_X,\nu)$ of $(Y,\fs,\eta)$ (whose existence is from \cite[Proposition 28.1.2]{kronheimer2007monopoles}), the bundle \[(P_{\det}\fs_X \times P_{SO(4)}X )|_\nu\] for the frame bundle $P_{SO(4)}X$ of $X$ has vanishing second relative Stiefel--Whitney class corresponding to $\tau(\eta)$, namely
\[
w_2((P_{\det}\fs_X \times P_{SO(4)}X) |_\nu, \tau^\fs(\eta) ) = 0.
\]
For simplicity, we will omit the spin$^c$ structure $\fs$ for framings and write \[\tau(\eta)=(\tau_{\det}(\eta),\tau_N(\eta),\tau_T(\eta))\]for the union of the framings for all spin$^c$ structures on $Y$.
\edefn
\brem\label{rem: admissible 2eta}
For any framing $\tau(\eta)$ of $\eta$, the admissible condition always holds for $\tau(2\eta)$ because $w_2$ is defined mod $2$.
\erem
The following lemma shows that the admissible condition does not depend on the choice of the $4$-dimensional spin$^c$ filling and hence it is a condition of the $3$-dimensional data $(Y,\fs,\eta)$. 
\begin{lemma}
    Suppose the tuple $(Y,\fs,\eta)$ admits two $4$-dimensional spin$^c$ fillings $(X_0,\fs_{X_0},\nu_0)$ and $(X_1,\fs_{X_1},\nu_1)$ such that a spin$^c$-structure $\fs_X$ on $X=X_0\cup_{Y} X_1$ whose determinant line bundle restricts to $P_{\det}\fs_{X_0}$ and $P_{\det}\fs_{X_1}$. Then a framing $\tau^\fs(\eta)$ is admissible for one side if and only if it is admissible for the other side.
\end{lemma}
\begin{proof}

This phenomenon comes from the definition of the (relative) Stiefel--Whitney class: The obstructions of recursive extensions of the framing (over a subset, in our case is $\eta$) to higher dimensional cells. Hence
\begin{align*}
&w_2((P_{\det}\fs_{X_0} \times P_{SO(4)}X_0) |_{\nu_0}, \tau^\fs(\eta) ) + w_2((P_{\det}\fs_{X_1} \times P_{SO(4)}X_1) |_{\nu_1}, \tau^\fs(\eta) ) \\
&\equiv \#\{\text{obstructions on $\nu_0^{(2)}$ when extending $\tau^\fs(\eta)$} \} + \\
&\;\;\;\text{    } \#\{\text{obstructions on $\nu_1^{(2)}$ when extending $\tau^\fs(\eta)$} \}
\\
&\equiv w_2((P_{\det}\fs_X \times P_{SO(4)}X) |_\nu )\pmod 2.
\end{align*}
The spin$^c$ condition implies the vanishing of $w_2((P_{\det}\fs_X \times P_{SO(4)}X) |_\nu )$. Hence the first two terms in the above equation have the same parity. If they are both even, then the framing $\tau^\fs(\eta)$ is admissible for both sides. If not, then the framing is not admissible for both sides.
\end{proof}
With the notion of admissible framings on $1$-cycles, we can now define the cobordism map of local system between $1$-cycles equipped with admissible framings:
\bdefn\label{def:cob}
    Suppose $(W,\nu):(Y_0,\eta_0)\to (Y_1,\eta_1)$ is a cobordism, that is\begin{itemize}
        \item $W:Y_0\to Y_1$ is a cobordism between closed $3$-manifolds.
        \item $\eta_i\subset Y_i$ for $i=0,1$ is an integral (rather than rational) smooth  $1$-cycle.
        \item $\nu\subset W$ is an integral smooth relative $2$-chain such that $\partial \nu=\eta_1-\eta_0$.
    \end{itemize}
    Let $\tau(\eta_i)$ be an admissible framing of $\eta_i$. We write $\tau( \eta_1-\eta_0)$ to denote $\tau(-\eta_0)\cup \tau(\eta_1)$ (only) in the notation of relative characteristic classes for short. Note that $\tau(-\eta_0)$ and $\tau(\eta_1)$ are admissible implies $\tau( \eta_1-\eta_0)$ is admissible with respect to $(W,\nu)$, but the converse statement is not true in general.
\edefn
We use $\nu$ to define a $W$-morphism
\[
\Gamma_\nu: \Gamma^{\tau(\eta_0)}_{\eta_0} \to \Gamma^{\tau(\eta_1)}_{\eta_1}
\]
by the following. Let $\gamma=(A,s,\phi)$ be a representative of $[\gamma] \in \cB^\sigma(W)$ that connects $\bfB_0=[B_0,r_0,\psi_0] \in \cB^\sigma(Y_0)$ to $\bfB_1=[B_1,r_1,\psi_1] \in \cB^\sigma(Y_1)$. Let $\exp_u(\lambda)$ denote $u^{\lambda}$ and let $k\in\Z$. Then define
\begin{equation}\label{eq: cob map}
    \begin{aligned}
\Gamma_\nu([\gamma]): \Gamma_{\eta_0}^{\tau(\eta_0)} (\bfB_0) \to& \Gamma_{\eta_1}^{\tau(\eta_1)} (\bfB_1)\\
  \exp_u\big( \frac{1}{2} \frac{i}{2\pi}\int_{\tau_{\det}(\eta_0)} B_0^t \big)R \to& \exp_u\big(\frac{1}{2} \frac{i}{2\pi}\int_{\tau_{\det}(\eta_1)} B_1^t \big)R  \\
  \exp_u\big(\frac{1}{2} \frac{i}{2\pi}\int_{\tau_{\det}(\eta_0)} B_0^t + k\big) \mapsto &\exp_u\big(\frac{1}{2} \big(\frac{i}{2\pi}\int_{\tau_{\det}(\eta_0)} B_0^t + k \\&+\frac{i}{2\pi}\int_\nu F_{A^t}  +\topo(W,  \nu, \tau( \eta_1-\eta_0))\big)  \big)
\end{aligned}
\end{equation}
Here $\topo$ is the topological term:
\[
\topo(W, \nu, \tau( \eta_1-\eta_0)) = \langle e(N_\nu, \tau_N(\eta_1-\eta_0)), [\nu]\rangle +  \langle e(T_\nu, \tau_T( \eta_1-\eta_0)), [\nu]\rangle
\]
 where 
    \begin{itemize}
        \item $e(N_\nu, \tau_N(\eta_1-\eta_0))\in H^2(\nu, \partial \nu) $ is the relative Euler class of the normal bundle of $\nu$ mod $ \eta_1-\eta_0$ corresponding to $\tau_N( \eta_1-\eta_0)$.
        \item $e(T_\nu, \tau_T( \eta_1-\eta_0))\in H^2(\nu, \partial \nu)  $ is the relative Euler class of the tangent bundle of $\nu$ mod $ \eta_1-\eta_0$ corresponding to $\tau_T( \eta_1-\eta_0)$.
    \end{itemize}

\begin{example}\label{exmp: closed}
    If $\nu$ is closed, the integral of the curvature is the evaluation of the first Chern class
    \[\frac{i}{2\pi}\int_{\nu}F_{A^t}=\langle c_1(\fs),[\nu]\rangle,\]and
    the topological term is given by the self-intersection number and the Euler characteristic
    \[
\topo(W, \nu) = \nu\cdot \nu +  \chi(\nu).
\]If we take $-\nu$, then the curvature term has a minus sign but the topological term keeps the same. Hence if we take the product cobordism $W=Y\times I$ and a closed $\nu$, then \begin{equation}\label{eq: non id}
    HM(W;\Ga_{\nu})\circ HM(W;\Ga_{-\nu})\neq \id.
\end{equation}We will come back to this case in Remark \ref{rem: non id}.
\end{example}

To show that the above definition of the $W$-morphism between $1$-cycles equipped with admissible framings is well-defined, we need the following lemma.
\begin{lemma}\label{lem: even}Suppose $P_{SO(4)}$ is the frame bundle of the cobordism $W$ and $P_{\det}$ is the determinant line bundle of $\fs_W\in \spinc(W)$. Under the assumption that \[w_2(P_{\det} \times P_{SO(4)} |_\nu, \tau( \eta_1-\eta_0)) = 0,\]the sum
    \[
   \frac{i}{2\pi}\big(\int_{\tau_{\det}(\eta_0)} B_0^t +\int_\nu F_{A^t}  - \int_{\tau_{\det}(\eta_1)} B_1^t\big) +\topo(W,  \nu, \tau( \eta_1-\eta_0)) 
    \]
    is an even integer.
\end{lemma}
\begin{proof}
    The proof is divided into two steps. We first prove that 
    \[
    \frac{i}{2\pi}\big(\int_{\tau_{\det}(\eta_0)} B_0^t +\int_\nu F_{A^t} - \int_{\tau_{\det}(\eta_1)} B_1^t \big)
    \]
    is equal to the relative Euler number of the determinant line bundle. Then the second step is to show that the sum of this number and the topological term is even.

    The first step is a generalization of \cite[Proposition 1.34]{salamon2000spin}. Choose a generic section $s$ of $P_{\det}|_\nu$ such that on the boundary $s$ corresponds with $\tau_{\det}( \eta_1-\eta_0)$. Then $s$ has only nondegenerate zeros and they are all in the interior of $\nu$. Recall that a zero $x$ of $s$ is called nondegenerate if the map 
    \[
    \nabla s(x): T_x\nu \to P_{{\det},x}
    \]
is an isomorphism and in this case the index $\text{ind}(s,x) = \pm 1$ is determined by whether or not this isomorphism is orientation preserving. Note that the fiber $P_{{\det},x}$ carries a natural orientation as a complex vector space. By definition, the relative Euler number is 
    \[
    \langle e(P_{\det}|_\nu, \tau_{\det}( \eta_1-\eta_0)), [\nu]\rangle = \sum_{s(x)=0} \text{ind}(s,x).
    \]

    Choose a splitting 
    \[
    \nu = \Sigma_1\cup_C \Sigma_2
    \]
    where $\Sigma_2\subset \mathring{\nu}$ is a set of small disks around the zeros of $s$, and orient $ C \subset \partial\Sigma_1$ as the boundary of $\Sigma_1$. Then choose nonzero sections $s_i: \Sigma_i \to P_{\det}$ with $|s_i(x)| = 1$ and $s_1$ corresponds with $\tau_{\det}(\eta_0)$ and $\tau_{\det}(\eta_1)$ on $\partial \nu$. Define $\gamma: C\to S^1$ by 
    \[
    s_2(x) = \gamma(x)s_1(x), \quad  x\in C.
    \]
    
    Recall that $A^t$ is a connection in $P_{\det}$ that restricts to $B^t_0$ and $B^t_1$ on the boundary. Define $\alpha_i \in \Omega^1(\Sigma_i, i \R)$ by $\nabla_{A^t} s_i = \alpha_i s_i$. Then $F_{A^t}|_{\Sigma_i} = d\alpha_i$ and $\alpha_2|_C = \alpha_1|_C + \gamma^{-1}d\gamma$. It follows from Stokes' theorem that 
    \[\begin{aligned}
         \int_\nu F_{A^t} &= \int_{\Sigma_1} d\alpha_1 + \int_{\Sigma_2} d\alpha_2\\
         &= \int_{\eta_1 -\eta_0} \alpha_1 + \int_C (\alpha_1-\alpha_2)\\
         &=\int_{\tau_{\det}(\eta_1)} B_1^t  - \int_{\tau_{\det}(\eta_0)} B_0^t -\int_C\gamma^{-1}d\gamma \\
         &= \int_{\tau_{\det}(\eta_1)} B_1^t  - \int_{\tau_{\det}(\eta_0)} B_0^t - 2\pi i\deg(\gamma)\\
         &= \int_{\tau_{\det}(\eta_1)} B_1^t  - \int_{\tau_{\det}(\eta_0)} B_0^t - 2\pi i\sum_{s(x)=0} \text{ind}(s,x)
    \end{aligned}\]

We already proved that 
\[
\frac{i}{2\pi}\big(\int_{\tau_{\det}(\eta_0)} B_0^t +\int_\nu F_{A^t} - \int_{\tau_{\det}(\eta_1)} B_1^t\big) =    \langle e(P_{\det}|_\nu, \tau_{\det}( \eta_1-\eta_0)), [\nu]\rangle.
\]

Now we start the second step of the proof. In the closed case, the top Stiefel--Whitney class is the top Euler class mod $2$ (see \cite[Proposition 1.4.9]{Kirbycal1999}). In the relative case, one still has similar reason: An obstruction to extend a framing on the boundary to a framing on a $2$-cell, corresponds to a zero of the extension of the section given by the first direction of the boundary framing. So
\[
w_2(P_{\det}|_\nu, \tau_{\det}( \eta_1-\eta_0)) \equiv e(P_{\det}|_\nu, \tau_{\det}( \eta_1-\eta_0))\pmod 2.
\]
The relative Stiefel--Whitney class also holds the Whitney duality (see \cite[Theorem 4.1]{Kervaire57}, where the formula might involve closed Stiefel--Whitney classes, but here we have assumed sufficient framings on the boundary such that all terms are relative classes):
\[\begin{aligned}
    &w_2(P_{\det} \times P_{SO(4)} |_\nu, \tau(\eta_1-\eta_0)) = \\
&w_2(P_{\det}  |_\nu, \tau_{\det}( \eta_1-\eta_0)) + w_2( P_{SO(4)} |_\nu, (\tau_N \times \tau_T)(\eta_1-\eta_0) 
+\\
&w_1(P_{\det}  |_\nu, \tau_{\det}( \eta_1-\eta_0)) \cdot w_1(P_{SO(4)} |_\nu, (\tau_N \times \tau_T)(\eta_1-\eta_0)
\end{aligned}\]
We assume that all bundles are orientable, so the extension of the boundary framing to $1$-cell would not have any obstruction. Hence the relative $w_1$ vanishes. Furthermore, we have
\[\begin{aligned}
    &   w_2(P_{\det} \times P_{SO(4)} |_\nu, \tau(\eta_1-\eta_0))  \\
&=w_2(P_{\det}  |_\nu, \tau_{\det}( \eta_1-\eta_0)) + w_2( P_{SO(4)} |_\nu, (\tau_N \times \tau_T)(\eta_1-\eta_0)
\\
&=w_2(P_{\det}  |_\nu, \tau_{\det}( \eta_1-\eta_0)) +w_2(N_\nu, \tau_N(\eta_1-\eta_0)) +  w_2( T_\nu ,  \tau_T(\eta_1-\eta_0)).
\end{aligned}\]
So the assumption $w_2(P_{\det} \times P_{SO(4)} |_\nu, \tau(\eta_1-\eta_0)) = 0$ implies that 
\[
e(P_{\det}  |_\nu, \tau_{\det}( \eta_1-\eta_0)) \equiv e(N_\nu, \tau_N(\eta_1-\eta_0)) +  e( T_\nu ,  \tau_T(\eta_1-\eta_0))\pmod 2.
\]
As the topological term $\topo(W,\nu, \tau(\eta_1-\eta_0)) $ is the evaluation of the right-hand-side, we complete the second step of the proof.
\end{proof}

The following example is the simplest case of the cobordism map: the product cobordism. It also shows what will happen when the framings of loops are changed.
\begin{example}\label{exmp: product cob}
Consider a product cobordism $(W,\nu) = (Y,\eta)\times I:(Y,\eta)\to (Y,\eta)$. In the construction of the cobordism map, the only solutions are constant paths (namely on each slice of $Y \times I$ the solution is $B_0^t$). Hence the curvature term vanishes and $B_0^t  = B_1^t $.


The topological term depends on the normal and tangent framings on two sides, whose difference can be measured by the zeros on $\nu$ (or equivalently, the relative Euler class of $N\vert _\nu$ or $T\vert _\nu$). The difference between determinant framings on two sides can be measured by the relative Euler class of $P_{\det}\vert _\nu$. Lemma \ref{lem: even} implies that the sum of these two contributions are even and then the cobordism map associated to $(W = Y \times I,\nu)$ for $\tau( \eta_1-\eta_0)$ is $u^m\cdot \id$ for some $m\in\Z$. 

Here we use the canonical choice of homology orientation for product cobordism from \cite[Proposition 8.3]{scaduto2015instantons} described as follows. Let $\al$ be any orientation of $H_1(Y)$ which induces an orientation $\be$ of $H_1(Y\times I)$. Then the canonical homology orientation is\[\mu_Y^{\id}=(-1)^{\frac{b_1(Y)^2+b_1(Y)}{2}}\be\wedge \al.\]Note that $\overline{\mu}_Y^{\id}=\mu_Y^{\id}$ for the automorphism $\mu\mapsto \overline{\mu}$ described at the end of \cite[\S 8.2]{scaduto2015instantons}.

\end{example}
The product cobordism provides canonical isomorphisms for different choices of admissible framings. Recall that $R=\Z[u^{\pm 1}]$. We obtain an object
\[
{HM}(Y; \Gamma_\eta):= \{{HM}(Y; \Gamma_\eta^{\tau(\eta)}),\tau(\eta) \text{ admissible} \} 
\]
in $R\text{-}\mathrm{MOD}/\mathrm{CAN}$ from \cite[Page 453]{kronheimer2007monopoles}, or equivalently a transitive system of $R$-modules (see \cite[Definition 1.1]{Juhasz2012}). There is a functor from $R\text{-}\mathrm{MOD}/\mathrm{CAN}$ to $R\text{-}\mathrm{MOD}$ via the colimit. Hence we obtain the following result.
\bprop[{\cite[Theorem 23.3.4]{kronheimer2007monopoles}}]\label{prop: HM}
The monopole Floer homology with our local system induces a covariant functor
\[\begin{aligned}
HM(-;\Gamma_{-}):\mathrm{COB}\text{-}\mathrm{L} &\to R\text{-}\mathrm{MOD}\\
(Y,\eta) &\mapsto HM(Y; \Gamma_{\eta}) := \colim \{HM(Y; \Gamma_\eta^{\tau(\eta)}),~\tau(\eta)~\text{admissible} \} \\(W,\nu)&\mapsto HM(W;\Ga_\nu)
\end{aligned}\]

Here $\mathrm{COB}\text{-}\mathrm{L}$ consists of objects $(Y, \eta)$ and morphisms $(W, \nu)$ from Definition \ref{def:cob}, where $Y$ and $W$ are connected and possibly empty. We call $\eta,\nu$ \emph{local system sets}.


\eprop
\brem\label{rem: disconnected}
Monopole theory for disconnected closed $3$-manifolds and cobordism maps between them could be defined, but needs extra care on the versions of each component and the orientations of the moduli spaces. For the first issue, from Lin \cite[\S 3]{pin2monopole} and Bloom \cite{Bloom2013morse}, one expects that for a cobordism $W:\bigsqcup_{i=0}^nY_i\to Y_{n+1}$ for connected $3$-manifolds $Y_j$ with $j=1,\dots,n+1$, there are cobordism maps over $\Z/2$\[\widecheck{HM}_\bu(Y_0)\ot \widehat{HM}_\bu(Y_1)\ot\cdots\ot \widehat{HM}_\bu(Y_n)\to \widecheck{HM}_\bu(Y_{n+1}),\]\[\widehat{HM}_\bu(Y_0)\ot \widehat{HM}_\bu(Y_1)\ot\cdots\ot \widehat{HM}_\bu(Y_n)\to \widehat{HM}_\bu(Y_{n+1}).\]For the second issue, Kronheimer--Mrowka \cite[\S 2.5 and \S 2.6]{kronheimer2010knots} worked out the case of irreducible connections and the cobordism maps for non-torsion spin$^c$ structures, for which $\overline{HM}_\bu$ vanishes and $\widecheck{HM}_\bu\cong \widehat{HM}_\bu$, and there is no need to consider the versions. In the proof of surgery exact triangle, when stretching some metric on the cobordism to the infinity, some piece of the cobordism could have disconnected boundary (see \cite[Figure 19]{Bloom2013morse}). However, we always consider the union of all pieces of the cobordism and count solutions of broken trajectories to avoid the cobordism map for multiple boundary ends; for example, see the proof of \cite[Proposition 5.2]{kronheimer2007monopolesandlens} and \cite[Equation (19)]{kronheimer2007monopolesandlens}.
\erem
\subsection{Cobordism maps for non-orientable surfaces}
If the $2$-chain $\nu \subset W$ in the definition of the cobordism map (Definition \ref{def:cob}) is non-orientable with respect to the orientation of $\partial \nu$, we can still define the cobordism map if we add extra data compatible with the orientation of $\partial \nu$. The construction of such cobordism maps is necessary when we use the homological trick, and the essential result we need is in Example \ref{exmp: non-orientable}.
\bdefn\label{def:cob-non-orientable}
    We extend the definition of a cobordism to\[(W,\wti{\nu}=(\nu,\kappa,\tau(\kappa))):(Y_0,\eta_0)\to (Y_1,\eta_1)\]where $\wti{\nu}$ is called an \emph{enhanced local system set} and we have the following\begin{itemize}
        \item $W:Y_0\to Y_1$ is a cobordism between closed $3$-manifolds.
        \item $\eta_i\subset Y_i$ for $i=0,1$ is an integral (rather than rational) smooth $1$-cycle. Note that $\eta_i$ is oriented by definition of $1$-cycle.
        \item $\nu\subset W$ is the union of smooth embedded, possibly non-orientable surfaces and $\kappa \subset \nu\subset W$ is a smooth $1$-cycle, such that $\nu \backslash \kappa$ is a smooth $2$-cycle (in particular, oriented) with $\partial(\nu \backslash \kappa) = \eta_1 -\eta_0 + 2\kappa$.
        \item $\tau(\kappa)$ is a framing of $\kappa$. Due to Remark \ref{rem: admissible 2eta}, when the framing $\tau(\eta_i)$ on $\eta_i$ is admissible, we always have the admissible condition of $\tau(\eta_1 -\eta_0 + 2\kappa)$, i.e.\ 
        \[w_2(P_{\det} \times P_{SO(4)} |_{\nu \backslash \kappa}, \tau( \eta_1-\eta_0+ 2\kappa)) = 0\]for the frame bundle $P_{SO(4)}$ of the cobordism $W$ and the determinant line bundle $P_{\det}$ of $\fs_W\in\spinc(W)$.
        
    \end{itemize}


\edefn
\begin{remark}
If $\nu$ is orientable with respect to the orientation of $\eta_0$ and $\eta_1$, one can just take $\kappa = \emptyset$. Then Definition \ref{def:cob-non-orientable} reduces to Definition \ref{def:cob}. In general, if $\tau_T(\partial \Sigma)$ is the outward normal tangent vector field of $\partial \Sigma$, then the homology class $[\kappa]$ is the Poincar\'{e} dual of the first relative Stiefel--Whitney class $w_1(T\Sigma , \tau_T(\partial \Sigma))$, i.e.\ $\kappa$ measures the non-orientability of $\nu$.
\end{remark}
We use $\wti{\nu}$ to define a $W$-morphism
\[
\Gamma_{\wti{\nu}}: \Gamma^{\tau(\eta_0)}_{\eta_0} \to \Gamma^{\tau(\eta_1)}_{\eta_1}
\]
by the following. Let $\gamma=(A,s,\phi)$ be a representative of $[\gamma] \in \cB^\sigma(W)$ that connects $\bfB_0=[B_0,r_0,\psi_0] \in \cB^\sigma(Y_0)$ to $\bfB_1=[B_1,r_1,\psi_1] \in \cB^\sigma(Y_1)$. Let $\exp_u(\lambda)$ denote $u^{\lambda}$ and let $k\in\Z$. Then define
\begin{equation}\label{eq: cob map 2}
    \begin{aligned}
\Gamma_{\wti{\nu}}([\gamma]): \Gamma_{\eta_0}^{\tau(\eta_0)} (\bfB_0) \to& \Gamma_{\eta_1}^{\tau(\eta_1)} (\bfB_1)\\
  \exp_u\big( \frac{1}{2} \frac{i}{2\pi}\int_{\tau_{\det}(\eta_0)} B_0^t \big)R \to& \exp_u\big(\frac{1}{2} \frac{i}{2\pi}\int_{\tau_{\det}(\eta_1)} B_1^t \big)R  \\
  \exp_u\big(\frac{1}{2} \frac{i}{2\pi}\int_{\tau_{\det}(\eta_0 )} B_0^t + k\big) \mapsto &\exp_u\big(\frac{1}{2} \big(\frac{i}{2\pi}\int_{\tau_{\det}(\eta_0)} B_0^t -\frac{i}{2\pi}\int_{\tau_{\det}( 2\kappa)} \gamma  + k \\&+\frac{i}{2\pi}\int_{\nu \backslash \kappa} F_{A^t}  +\topo(W,  \nu \backslash \kappa, \tau( \eta_1-\eta_0 + 2\kappa))\big)  \big)
\end{aligned}\end{equation}
Here $\topo$ is the topological term:
\[
\topo(W, \nu \backslash \kappa, \tau( \eta_1-\eta_0+ 2\kappa)) = \langle e(N_{\nu \backslash \kappa}, \tau_N(\eta_1-\eta_0+ 2\kappa)), [\nu]\rangle +  \langle e(T_{\nu \backslash \kappa}, \tau_T( \eta_1-\eta_0+ 2\kappa)), [\nu]\rangle
\]

To show that the above definition is well-defined, we need the following lemma. The statement and the proof are similar to those of Lemma \ref{lem: even}.
\begin{lemma}\label{lem: even2}Suppose $P_{SO(4)}$ is the frame bundle of the cobordism $W$ and $P_{\det}$ is the determinant line bundle of $\fs_W\in \spinc(W)$. Under the assumption that \[w_2(P_{\det} \times P_{SO(4)} |_{\nu\backslash\kappa}, \tau( \eta_1-\eta_0+2\kappa)) = 0,\]the sum
    \[
   \frac{i}{2\pi}\big(\int_{\tau_{\det}(\eta_0)} B_0^t- \int_{\tau_{\det}(2\kappa)} \gamma +\int_{\nu \backslash \kappa} F_{A^t}  - \int_{\tau_{\det}(\eta_1)} B_1^t\big) +\topo(W,  {\nu \backslash \kappa}, \tau( \eta_1-\eta_0 + 2\kappa)) 
    \]
    is an even integer.
\end{lemma}

\begin{example}\label{exmp: non-orientable}
    Suppose $W=Y\times I:Y_0\to Y_1$ is the product cobordism. Let $\eta_0 = 2\eta+\eta_1\subset Y_0$, where $2\eta$ consists of two copies of the same $1$-cycle. Suppose $\nu = \eta \times I'+\eta_1\times I$, where $I'=[0,1]$ and two boundary components of $\eta \times I'$ are on $Y_0$. Then $\eta \times I'$ is non-orientable with respect to the orientation of $\eta_0$. We take \[ \kappa = \eta \times \{\frac{1}{2}\}.\]Then $(\eta \times I')\backslash\kappa$ is the union of two annuli and there is an orientation of $(\eta \times I')\backslash\kappa$ such that $\partial(\nu\backslash\kappa)=-2\eta + 2\kappa$. There is a standard choice of $\tau(\kappa)$ induced from $\tau(\eta)$. Suppose $\wti{\nu}=(\nu,\kappa,\tau(\kappa))$. Then as a variant of Example \ref{exmp: product cob}, the corresponding cobordism map \[HM(W;\Ga_{\wti{\nu}}):HM(Y;2\eta+\eta_1)\to HM(Y;\eta_1)\]is an isomorphism. If we choose different admissible $\tau(\kappa)$, the corresponding cobordism map will change by the multiplication of $u^m$ for some $m\in\Z$. From now on, we always choose the standard framing of $\kappa$.

    Furthermore, suppose $\eta_0=2\eta+2\eta'+\eta_1$ and suppose $\wti{\nu}$ and $\wti{\nu}'$ are the corresponding enhanced local system sets. Then we have\[HM(W;\Ga_{\wti{\nu}})\circ HM(W;\Ga_{\wti{\nu}'})=HM(W;\Ga_{\wti{\nu}'})\circ HM(W;\Ga_{\wti{\nu}})=HM(W;\Ga_{\wti{\nu}+\wti{\nu}'}).\] Similar to the discussion after Example \ref{exmp: product cob}, we obtain a transitive system of $R$-modules \[HM(Y;\Ga_{\ov{\eta}}):=\{HM(Y;\Ga_{\eta}^{\tau(\eta)}),\eta-\ov{\eta}=2\eta',~\tau(\eta)~\text{admissible}\}\]
\end{example}
\subsection{Properties}
In this subsection, we study the dependence of the local system set when $u = -1$, in comparison with the discussion at the end of \S \ref{subsec: Kronheimer--Mrowka's construction} on the Kronheimer--Mrowka local system, as well as \cite[\S 2]{LY2025dimension} and \cite[\S 2]{Ye2025dualknot} on instanton Floer homology. For simplicity, we abuse the notation and use the local system set $\eta,\nu$ to denote the corresponding local system with $u=-1$. 
\begin{lemma}\label{lem: cob map depending on H_2}
    
	Let\[(W,\nu):(Y_0,\eta_0)\to (Y_1,\eta_1)\aand (W,\nu'):(Y_0,\eta_0')\to (Y_1,\eta_1')\]be two cobordisms such that $\nu'-\nu=2\rho+\partial\theta$ for some smooth relative $2$-chain $\rho$ and smooth $3$-chain $\theta$ in $W$, then we have\[HM(W;\nu')=(-1)^{\frac{\chi(\partial\theta)}{2}}HM(W;\nu)\]under the canonical isomorphisms in Example \ref{exmp: non-orientable}. Note that $\chi(\partial\theta)$ is always even.
    
\end{lemma}
\bpf
By functoriality of the cobordism map, we can consider the two terms $2\rho$ and $\partial \theta$ separately. 

When $\theta=\emptyset$, the factor of $2$ before $\rho$ makes the exponential term in the cobordism map \eqref{eq: cob map} becomes an even integer. Hence $u=-1$ and the choice of canonical isomorphisms in Example \ref{exmp: non-orientable} implies that it does not affect the cobordism map. 

When $\rho=\emptyset$, for the exponential term in \eqref{eq: cob map}, the curvature term vanishes due to the Stokes' theorem, and the topological term is computed in Example \ref{exmp: closed}.
\epf
\brem\label{rem: non id}
    From Lemma \ref{lem: cob map depending on H_2}, the composition in \eqref{eq: non id} becomes $(-1)^{\chi(\nu)} \id$ rather than $\id$ when $u=-1$. This is because $\nu\cup(-\nu)=\partial(\nu\times I)$. This non-additive result of $2$-chain should be considered carefully, so sometimes we prefer to use $\cup$ rather than $+$ to denote the union of $2$-chains.
\erem
\begin{lemma}\label{lem: I_S is an iso}
	Suppose $Y$ is a closed oriented $3$-manifold. Suppose $\eta_1,\eta_2\subset Y$ are two local system sets such that $[\eta_1] = [\eta_2] \in H_1(Y;\Z/2)$. Then there exists a $1$-cycle $\eta$ and $2$-chain $S$ in $Y$ with $\partial S = \eta_1-\eta_2+2\eta$. Let \[S'=S\times\{0\}\cup \eta_1\times I\aand S''=-S\times\{0\}\cup(\eta_2-2\eta)\times I\subset Y\times I\]and let the corresponding cobordism maps by \[\bI_{S}'=HM(Y\times I;S' ):HM(Y;\eta_2)\cong HM(Y;\eta_2-2\eta)\to HM(Y;\eta_1)\aand \]\[ \bI_{S}''=HM(Y\times I;S'' ):HM(Y;\eta_1)\to HM(Y;\eta_2-2\eta)\cong HM(Y;\eta_2),\]where the canonical isomorphisms are from Example \ref{exmp: product cob}. Then we have the following.
	\begin{itemize}
		\item For any choice of $S$, the maps $\bI_S'$ and $\bI_{S}''$ are isomorphisms.
		\item If $H_2(Y;\Z) = 0$ and $[\eta_1] = [\eta_2] \in H_1(Y;\Z)$, then, up to sign, $\bI_S'$ is independent of the choice of $S$. As a consequence, up to sign, the $\Z$-modules $HM(Y;\eta_1)$ and $HM(Y;\eta_2)$ are canonically isomorphic.
	\end{itemize}
\end{lemma}
\bpf
Both arguments follow from Lemma \ref{lem: cob map depending on H_2}. For the first argument, the local system set in the composition $\bI_{S}''\circ \bI_{S}'$ differs $(\eta_2-2\eta)\times [0,2]$ by $\partial (-S\times I)$. Hence the composition equals the identity up to sign. Similarly, the composition $\bI_{S}'\circ \bI_{S}''$ equals the identity up to sign. Hence both maps are isomorphisms.

For the second argument, we have $\eta=0$ because $[\eta_1] = [\eta_2] \in H_1(Y;\Z)$. Different choices of $S$ differ by a $2$-cycle in $Y$, which is the boundary of a $3$-chain by $H_2(Y;\Z) = 0$.
\epf
\brem\label{rem: enhanced level set}
From Example \ref{exmp: non-orientable}, the isomorphisms $\bI_S'$ and $\bI_S''$ in Lemma \ref{lem: I_S is an iso} are related to the enhanced local system set \[\wti{S}=(S\cup (\eta\times I),\kappa=\eta\times \{1/2\},\tau(\kappa)),\]where $S\cup (\eta\times I)$ is an unorientable $2$-chain with boundary $\eta_1-\eta_2$.
\erem
\section{Triangle over local system}\label{sec: Triangle over local system}
In this section, we prove the surgery exact triangle over our local system in \S \ref{subsec: Our construction}. We set $u=-1$ in the local system throughout this section and use the local system sets $\eta,\nu$ to denote the corresponding local system. Then all monopole Floer homologies are $\Z$-modules (possibly with $U$-actions).

\bthm\label{thm: local system triangle}
Let $(Y_0,Y_1,Y_2)$ be a surgery triad and let $W_i:Y_i\to Y_{i+1}$ for $i\in\Z/3$ be the corresponding surgery cobordism. Let $\eta_i=K_i\subset Y_i$ be the attaching circle for the $2$-handle in $W_i$ and let $\nu_i\subset W_i$ be the union of core disk and cocore disk. Let $\omega \subset Y_i\backslash K_i$ be any (possibly empty) $1$-submanifold away from the surgery region. Then there exists an exact triangle
\[	\xymatrix{
	HM(Y_0;\eta_0\cup \omega)\ar[rr]^{HM(W_0;\nu_0\cup (\omega\times I))}&& HM(Y_1;\eta_1\cup \omega)\ar[dl]^{\quad\quad HM(W_1;\nu_1\cup (\omega\times I))}\\
	&HM(Y_2;\eta_2\cup \omega)\ar[lu]^{HM(W_2;\nu_2\cup (\omega\times I))\quad\quad}&
	}\]
\ethm
The proof basically follows \cite[\S 7]{Freeman2021triangle}, where the cases of \begin{equation}\label{eq: three versions}HM=\widehat{HM}_\bu,~\widecheck{HM}_\bu\aand\overline{HM}_\bu
\end{equation}are treated similarly. The case of $\widetilde{HM}_\bu$ is a special case in the proof of Proposition \ref{prop: ss local system} for $l=1$, which is based on the proof of $\widecheck{HM}_\bu$. So we omit this case in this section.

In this section, we use $HM$ to denote any fixed version in \eqref{eq: three versions}. The proof of Theorem \ref{thm: local system triangle} follows the standard approach based on the triangle detection lemma over $\Z$ (cf.\ \cite[Lemma 7.1]{kronheimer2011khovanov}, \cite[Lemma 5.1]{scaduto2015instantons}, and \cite[Proposition 7.1]{Freeman2021triangle}). We construct the first homotopy and the second homotopy via cobordism maps for some families of metrics (see the proof of \cite[Proposition 25.3.8]{kronheimer2007monopoles} for the construction of cobordism maps with families). We omit details in the standard construction and point out the main difference for our local system, assuming that the readers are familiar with the proof of the monopole surgery exact triangle over $\Z/2$ in \cite[\S 5]{kronheimer2007monopolesandlens}.
\subsection{The first homotopy}

The first homotopy is a chain homotopy from the double composition \[CM(W_{i+1};\nu_{i+1}\cup (\omega\times I)) \circ CM(W_i;\nu_i\cup (\omega\times I))\]to the zero map. This phenomenon comes from the appearance of an $\ov{\CP^2} \backslash \mathring{B}^4$ in the union $W_i \cup W_{i+1}$. Since the construction is cyclic, without loss of generality, we assume $i=0$. We start with a toy case, where the local system sets are simpler. Note that we use $\nu_i$ to denote different local system sets in the following propositions.
\bprop\label{prop: first homotopy example}


The composition of cobordisms $W_0\cup_{Y_1}W_1$ contains a sphere with self-intersection number $-1$. We denote a collar neighborhood of this sphere by $Z_1$, which is diffeomorphic to $\ov{\CP^2} \backslash \mathring{B}^4$. This sphere contains two semi-spheres: the cocore disk $\nu_0$ in $W_0$, and the core disk $\nu_1$ in $W_1$. Then the cobordism map\[HM(W_0\cup_{Y_1}W_1;\nu_0 \cup \nu_1)\]with our local system vanishes. Moreover, there is a chain homotopy from\[CM(W_0\cup_{Y_1}W_1;\nu_0 \cup \nu_1)\]to the zero map.
\eprop
\begin{proof}
We set:
\begin{itemize}
    \item $\eta_0 =\emptyset$ a cycle on $Y_0$;
    \item $\eta_{10}=\partial\nu_0 $ a cycle of $\{0\}\times Y_1$, with $\tau_N(\eta_{10}) =\gamma_{0-1}= \gamma_2$;
    \item $\eta_{11}=\partial\nu_1$ a cycle of $\{1\}\times Y_1$, with $\tau_N(\eta_{11}) = \gamma_2$;
    \item $\eta_2 =\emptyset$ a cycle on $Y_2$;
    \item $\tau_{T}(\eta_{10})$ and $\tau_{T}(\eta_{11})$ are given by the tangent vector of $\eta_{10}$ and $\eta_{11}$.
\end{itemize}
Then we have 
\[
\langle e(T_{\nu_0}, \tau_T(\eta_{10})),[\nu_0]\rangle = 1 = \langle e(T_{\nu_1}, \tau_T(\eta_{11})),[\nu_1]\rangle 
\]
\[
\langle e(N_{\nu_0}, \tau_N(\eta_{10})),[\nu_0]\rangle = -1  
\]
\[
\langle e(N_{\nu_1}, \tau_N(\eta_{11})),[\nu_1]\rangle  = 0.
\]
To satisfy the admissible condition
\[
\langle e(P_{{\det}}|_{\nu_0}, \tau_{\det}(\eta_{10})),[\nu_0]\rangle \equiv 0 \quad \text{mod } 2 
\]
\[
\langle e(P_{{\det}}|_{\nu_1}, \tau_{\det}(\eta_{11})),[\nu_1]\rangle \equiv 1\quad \text{mod } 2. 
\]

With the above setting, the framing $(\tau_N(\eta_{10}),\tau_T(\eta_{11}))$ corresponds to $(\tau_N(\eta_{11}),\tau_T(\eta_{11}))$. If $\fs$ is a spin$^c$-structure on the composition $W_0\cup_{Y_1}W_1$, which implies that $\langle c_1(\fs), [\nu_0\cup\nu_1]\rangle$ is odd, then an admissible framing $\tau_{\det}(\eta_{10})$ must correspond to an admissible framing $\tau_{\det}(\eta_{11})$. In summary, we have the following table for the zeros given by the framings:
\[
\begin{array}{c|c|c|c}
\#(\text{zeros}) & \nu_0 & \nu_1 & \nu_0 \cup \nu_1 \\ \hline \tau_T & 1 & 1 & 2 \\ \hline
\tau_N & -1 & 0 & -1 \\ \hline \text{possible } \tau_{\det} & 0 & 1 & 1 \\ \hline
\text{possible } \tau_{\det} & -2 & 3 & 1 \\ \hline \text{possible } \tau_{\det} & 0 & 3 & 3
\end{array}
\]
We see that the topological term is $2-1 = 1$, while the evaluation of the spin$^c$-structure may change if we change the determinant framings. In Definition \ref{def:cob} we have the formula for this composition $\Gamma_{\nu_0 \cup \nu_1}$:
\[\begin{aligned}
  \exp_u\big( {\frac{1}{2} \frac{i}{2\pi}\int_{\tau_{\det}(\eta_0)} B_0^t }\big)R \to& \exp_u\big({\frac{1}{2} \frac{i}{2\pi}\int_{\tau_{\det}(\eta_2)} B_2^t } \big)R  \\
  \exp_u\big({\frac{1}{2} \frac{i}{2\pi}\int_{\tau_{\det}(\eta_0)} B_0^t + k}\big) \mapsto& \exp_u\big(\frac{1}{2} \big(\frac{i}{2\pi}\int_{\tau_{\det}(\eta_0)} B_0^t +k\\&  +\frac{i}{2\pi}\int_{\nu_0\cup\nu_1} F_{A^t}+\topo(W,  {\nu_0\cup\nu_1}, \tau(\eta_2-\eta_0))\big)  \big)
\end{aligned}\]
Now we compare the image of the base point of the left-hand-side, with the base point of the right-hand-side (the difference between them is the degree shift of $u$ in the cobordism map):
\begin{equation}\label{equ:basepoint-difference}
\frac{1}{2} \big(\frac{i}{2\pi}\int_{\tau_{\det}(\eta_0)} B_0^t +\frac{i}{2\pi}\int_{\nu_0\cup\nu_1} F_{A^t}  +\topo(W,  \nu_0\cup\nu_1, \tau(\eta_2-\eta_0))\big)  -{\frac{1}{2} \frac{i}{2\pi}\int_{\tau_{\det}(\eta_2)} B_2^t }
\end{equation}
Here we have set $\eta_0 = \eta_2 = \emptyset$. Since we have 
\[
\frac{i}{2\pi}\int_{\nu_0\cup\nu_1} F_{A^t}  = \langle c_1(\fs), [\nu_0\cup\nu_1]\rangle,
\]
the difference (degree shift of $u$) is 
\begin{equation}\label{equ:difference on index for punctured CP2 cobordism}
\frac{1}{2}(\langle c_1(\fs), [\nu_0\cup\nu_1]\rangle + 1).
\end{equation}

We want to show the cobordism map already vanishes on a neighborhood of $\nu_0\cup\nu_1$. Recall that we denote it by $Z_1$ and $Z_1 =\ov{\CP^2} \backslash \mathring{B}^4 \subset W_0\cup_{Y_1}W_1$.

The cobordism map for one spin$^c$-structure with the trivial local system $\Gamma_\emptyset = \Z$ is given by (see \cite[Example 2.5]{Freeman2021triangle})
\[
\widehat{HM}(\ov{\CP^2} \backslash \mathring{B}^4, \fs_m)(e_j) = e_{j-\frac{m^2-1}{4}},
\]
where $j$ is a nonpositive even number, $\fs_m$ is the spin$^c$-structure that evaluates an odd number $m$ at the exceptional sphere $S^2 \simeq E_1 \subset \overline{\CP^2}$. Hence for the trivial local system $\Gamma_\emptyset = \Z$, the cobordism map of monopole Floer homology is the sum over all spin$^c$-structures:
\[
\widehat{HM}(\ov{\CP^2} \backslash \mathring{B}^4; \Gamma_\emptyset)(e_0) = 2e_0 + 2e_{-2} + 2 e_{-6} + \cdots.
\]
Observe that this vanishes over $\Z/2$ coefficients. On the other hand, by \eqref{equ:difference on index for punctured CP2 cobordism} the cobordism map for $\widehat{HM}$ with our local system is 
\begin{equation}
\widehat{HM}(\ov{\CP^2} \backslash \mathring{B}^4; \Gamma_{\nu_0\cup\nu_1})(e_0) = u^0e_0  +u^{-1}e_0 + u^1 e_{-2} + u^{-2}e_{-2} +u^2 e_{-6} + u^{-3}e_{-6} + \cdots.
\end{equation}
When $u=-1$, this map vanishes over $\Z$ coefficients. 

For the behavior on the chain level, we follow the argument in \cite[\S 5.1]{kronheimer2007monopolesandlens}. The boundary of $Z_1 = \ov{\CP^2} \backslash \mathring{B}^4$ is $S^3$, and we label critical points in $\cB^\sigma(S^3)$ as $\mathfrak{a}_{\lambda_i}$, where $\lambda_i$ are the eigenvalues of a self-adjoint Fredholm operator obtained as a small perturbation of the Dirac operator on $S^3$. For each critical point, we have a moduli space $M_z(Z_1, \mathfrak{a}_{\lambda_i})$. The choice of $z$ is equivalent to a choice of spin$^c$-structure $\fs$ on $Z_1$. Again, we use $\fs_m$ to denote the spin$^c$-structure that evaluates an odd number $m$ at the exceptional sphere $S^2 \simeq E_1 \subset Z_1$. In \cite[Lemma 5.3]{kronheimer2007monopolesandlens}, they showed that
\begin{itemize}
\item For $i\ge 0$, $M_z(Z_1, \mathfrak{a}_{\lambda_i})$ is empty.
\item For $i< 0$, $M_z(Z_1, \mathfrak{a}_{\lambda_i})$ contains a single point if the formal dimension is zero.
\item $M_{\fs_{2k-1}}(Z_1, \mathfrak{a}_{\lambda_i})$ and $M_{\fs_{1-2k}}(Z_1, \mathfrak{a}_{\lambda_i})$ have the same formal dimensions.
\end{itemize}
By \eqref{equ:chain-level-differential} and \eqref{equ:difference on index for punctured CP2 cobordism}, a point in $M_{\fs_{2k-1}}(Z_1, \mathfrak{a}_{\lambda_i})$ contributes to $k$-degree-shift and a point in $M_{\fs_{1-2k}}(Z_1, \mathfrak{a}_{\lambda_i})$ contributes to $(1-k)$-degree-shift, if the formal dimension is zero. The sum of $u^k$ and $u^{1-k}$ vanishes when $u=-1$, so the cobordism map for $CM$ with our local system is chain homotopic to zero. 
\end{proof}

Now we proceed to prove the double composite of cobordism maps in Theorem \ref{thm: local system triangle} is trivial. The only difference from the previous example is that, there are two additional semi-spheres in the local system sets.
\bprop\label{prop: first homotopy}

In the setup of Proposition \ref{prop: first homotopy example}, let $\nu_{-1}$ be the core disk in $W_0$, and let $\nu_{2}$ be the cocore disk in $W_1$. Then the cobordism map \[HM(W_0\cup_{Y_1}W_1; \nu_{-1} \cup \nu_0 \cup \nu_1 \cup \nu_{2})\]vanishes. Moreover, there is a chain homotopy from \[CM(W_0\cup_{Y_1}W_1; \nu_{-1} \cup \nu_0 \cup \nu_1 \cup \nu_{2})\]to the zero map.
\eprop

\begin{proof}

Note that the neighborhood $Z_1$ of $\nu_0\cup \nu_1$ is a disk bundle over $\nu_0\cup \nu_1$. Since $\nu_{-1}$ and $\nu_0$ intersects once, we can suppose that $\nu_{-1}\cap Z_1$ is a fiber disk $\nu_{-1}'$. Similarly, we suppose $\nu_2\cap Z_1$ is a fiber disk $\nu_2'$. Then by Lemma \ref{lem: cob map depending on H_2}, we have
\[\widehat{HM}(\ov{\CP^2} \backslash \mathring{B}^4; \nu_{-1}'\cup\nu_0\cup\nu_1\cup\nu_2')(e_0) = \widehat{HM}(\ov{\CP^2} \backslash \mathring{B}^4;\nu_0\cup\nu_1)(e_0)=0\]

The chain level argument is the same as the last step in the proof of Proposition \ref{prop: first homotopy example}. Note that the proof of Lemma \ref{lem: cob map depending on H_2} also works on the chain level.
\end{proof}

\subsection{The second homotopy}\label{subsec: second homotopy}
Following \cite[\S 5.2]{kronheimer2007monopolesandlens}, we consider a triple composite (see Figure \ref{fig:V1-KMOS})
 \begin{figure}
    \begin{Overpic}{\includegraphics[scale=0.7]{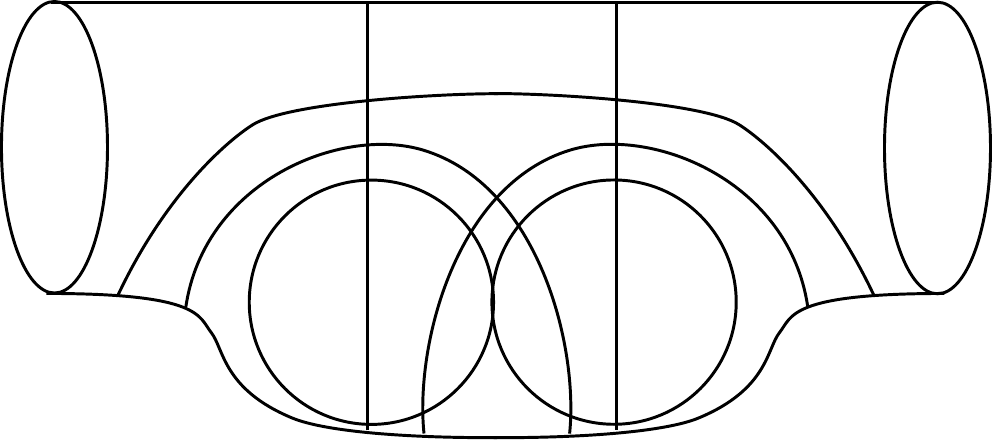}}   \put(10,41){$Y_1$}
    \put(87,41){$Y_1$}
    \put(32,41){$Y_2$}
    \put(63,41){$Y_3$}
    \put(48,36){$R_1$}
    \put(27,17){$E_1$}
    \put(69,17){$E_2$}
    \put(17,9){$S_1$}
    \put(80,9){$S_2$}
    \end{Overpic}
    \caption{The triple composite $V_1$ (cf.\ \cite[Figure 2]{kronheimer2007monopolesandlens}). This indicates the five $3$-manifolds which separate $V_1$.}
    \label{fig:V1-KMOS}
\end{figure}
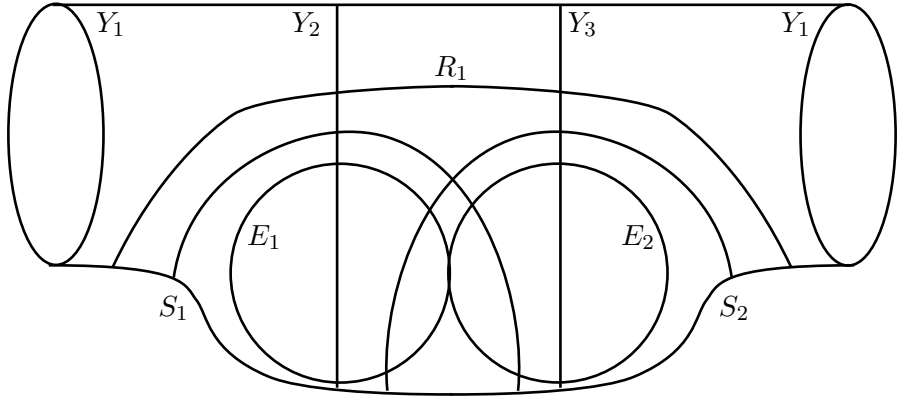
\[
V_1 = W_1\cup_{Y_2} W_2\cup_{Y_3}W_3.
\]
Inside it, there are two $(-1)$-spheres intersect transversally at one point:
\begin{itemize}
    \item $E_1$. The union of the cocore disk in $W_1$ and the core disk in $W_2$.
    \item $E_2$. The union of the cocore disk in $W_2$ and the core disk in $W_3$.
\end{itemize}
and five separating $3$-manifolds:
\begin{itemize}
    \item $S_i$. The boundary of the neighborhood of the $(-1)$-sphere $E_i$, $i=1,2$. Note that $S_i\cong S^3$. They intersect transversely in a $2$-torus.
    \item $R_1$. The boundary of the neighborhood of $E_1 \cup E_2$. Note that $R_1\cong S^1 \times S^2$ (see Figure \ref{fig:R1}).
    \item $Y_2$.
    \item $Y_3$.
\end{itemize}
 \begin{figure}
    \begin{Overpic}{\includegraphics[scale=0.6]{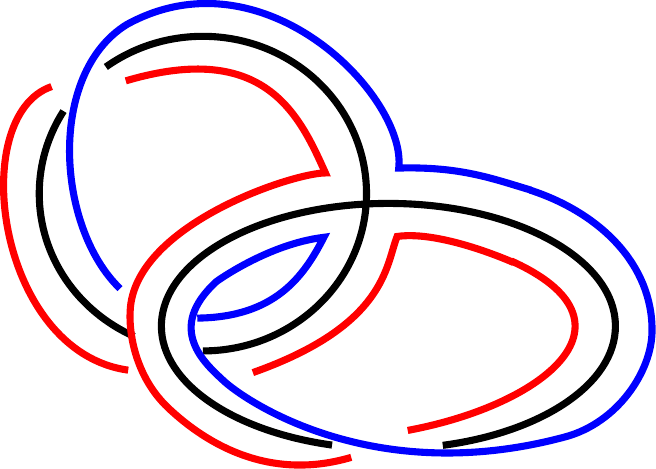}}   
    \end{Overpic}
    \caption{This figure is a low-dimensional description for $E_1\cup E_2$ and its neighborhood. Spheres $E_1$ and $E_2$ are depicted by black. They intersect at one point. $R_1$ is depicted by the colorful part.}
    \label{fig:R1}
\end{figure}
 Note that $V_1$ is a cobordism of $3$-manifolds so it induces a chain map. Following the assumption of Theorem \ref{thm: local system triangle}, we assume the local system set is the union of all three core disks and cocore disks. Let {$L$} be the chain homotopy on some version of monopole Floer homology induced by a $1$-parameter family of metrics on $V_1$ indexed by $t$, where:
 \begin{itemize}
    \item $R_1$ is stretched to $R_1\times [-\infty,\infty]$ for any $t$.
    \item $S_2$ is stretched to $S_2\times [-\infty,\infty]$ at $t=-\infty$.
    \item $S_1$ is stretched to $S_1\times [-\infty,\infty]$ at $t=\infty$.
\end{itemize}
We call this $1$-parameter family of metrics $Q(R_1)$, which is a notation in \cite[Figure 3]{kronheimer2007monopolesandlens}. To prove the second homotopy used for the triangle detection lemma (over $\Z$), all statements in \cite[Pages 77-78]{Freeman2021triangle} (see also \cite[\S 5.2]{kronheimer2007monopolesandlens}) go through except that $L$ induces isomorphisms in homology (cf.\ \cite[Proposition 5.6]{kronheimer2007monopolesandlens}).

In the rest of this subsection, we prove the following proposition, which is an analogous result of \cite[Lemma 7.10]{Freeman2021triangle}. Some steps in the proof are from \cite[\S 7.4]{Freeman2021triangle} but with modifications and more details. Together with Proposition \ref{prop: first homotopy}, Theorem \ref{thm: local system triangle} follows from the standard procedure based on the triangle detection lemma.
\bprop\label{prop:L-iso}
$L$ is an isomorphism.
\eprop

From \cite[Lemma 7.8]{Freeman2021triangle} (see also \cite[\S 5.3]{kronheimer2007monopolesandlens}), the condition whether the map $L$ is an isomorphism does not depend on $Y_1$, the knot $K$, or the surgery coefficient of $K$. Hence without loss of generality we can assume $Y_1= S^3$, $Y_2 = S^1\times S^2$ and $Y_3= S^3$. In this case $V_1$ is  twice-punctured $\CP^2\#2\ov{\CP^2}$, which is also diffeomorphic to twice-punctured $(S^2\times S^2)\#\overline{\CP^2}$; see Figures \ref{fig:single_E2_Kirby} and \ref{fig:single_Kirby}. The proof that $L$ is an isomorphism is divided into the following steps:
\begin{enumerate}
\item Choose a metric with positive scalar curvature for $V_1$, so that there are only reducible solutions. This is because the positive scalar curvature metric is closed under the connected sum operation for compact $4$-manifolds  \cite{SY,GL80scalar}.
\item  Find out spin$^c$-structures that have wall-crossing during $Q(R_1)$. It turns out that there are three types of such spin$^c$-structures.
\item Prove that the first two types contribute to zero and the third type contributes to the isomorphism.
\end{enumerate}

In the cobordism map of monopole Floer homology, we will fix a perturbing self-dual $2$-form $\omega$ for the cobordism (a $4$-manifold) of $3$-manifolds. During the $1$-parameter family $Q(R_1)$ of cobordisms (a family of $4$-manifolds), this perturbation is fixed. For the $4$-dimensional Seiberg-Witten equation, we have
\[
F^+_A  + \omega^+ = \sigma^+(\Phi).
\]
For a reducible solution, we have $\Phi = 0$. This happens when $\omega^+$ is on the wall:
\[
\{\eta \in \Omega^{2,+} | \exists A \in \mathcal{A}(\fs), F_A^+ + \eta = 0 \},
\]
where $\mathcal{A}(\fs)$ is the space of spin$^c$-connections for the spin$^c$-structure $\fs$.
\begin{lemma}
The set
\[
\{\eta \in \Omega^{2,+} | \exists A \in \mathcal{A}(\fs), F_A^+ + \eta = 0 \}
\]
is an affine space with codimension $b^+_2$.
\end{lemma}
During the $1$-parameter family $Q(R_1)$, the perturbation $\omega$ is fixed while the metric is changed. To find out the spin$^c$-structure that will cross the wall during $Q(R_1)$, we may want to eliminate the superscript $+$ of $F_A^+$, as the projection to the self-dual part depends on the metric. We use a reference self-dual vector perpendicular to the wall to project $2$-forms:

In our case $b^+_2=1$, so $H^{2,+}$ is a line and the complement of the wall has two components. For every orientation of $H^{2,+}$ there exists a unique (up to a real positive scalar coefficient) self-dual harmonic $2$-form
\[
\kappa_g \in H^{2,+}
\]
representing the given orientation of $H^{2,+}$. Then there exists a reducible solution if and only if
\begin{equation}\label{equ:wall}
(2\pi c_1(\fs) \cup [\kappa_g])[V_1]+ 4i\int_{V_1}\omega \wedge \kappa_g =0.
\end{equation}
Now we proceed to find out how the family of self-dual harmonic $2$-forms $\kappa_{g_t}$ looks like during $Q(R_1)$:

The triple composite $V_1$ has a Kirby diagram with three $2$-handles that represent $F$, $E_1$, $E_2$ (see Figure \ref{fig:V1_Kirby}). In this diagram, when we stretch $S_2$, the sphere $E_1$ would break off. A better Kirby diagram is compatible with the stretch. Hence we do a handle-slide to achieve the following diagram Figure \ref{fig:single_E2_Kirby}. It has three $2$-handles that represent $F$, $E_1 + E_2$, $E_2$, where $E_1 +E_2$ is a sphere that has no intersection with $E_2$ (and hence $S_2$), so it would not break off when we stretch $S_2$ (see Figure \ref{fig:slide_E1}). 
\begin{figure}
    \begin{Overpic}{\includegraphics[scale=0.6]{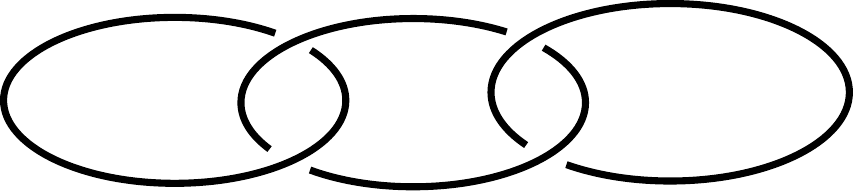}}
     \put(10,21){$0$}
     \put(0,0){$F$}
    \put(56,-3){$E_1$}
        \put(41,21){$-1$}
        \put(89,21){$-1$}
        \put(95,0){$E_2$}
        
    \end{Overpic}
    \caption{Kirby diagram of $V_1$}
    \label{fig:V1_Kirby}
\end{figure}

\begin{figure}
    \begin{Overpic}{\includegraphics[scale=0.6]{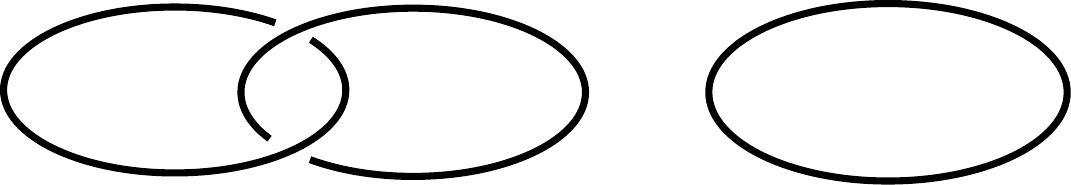}}
     \put(10,18){$0$}
     \put(0,0){$F$}
    \put(50,0){$E_1+E_2$}
        \put(41,18){$0$}
        \put(83,18){$-1$}
        \put(95,0){$E_2$}
        
    \end{Overpic}
    \caption{Kirby diagram compatible with stretching $S_2$}
    \label{fig:single_E2_Kirby}
\end{figure}
 \begin{figure}
    \begin{Overpic}{\includegraphics[scale=0.6]{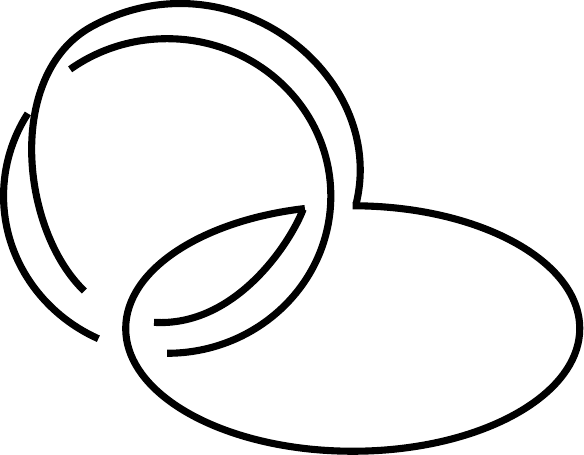}}   
     \put(-7,25){$E_2$}
    \put(90,0){$E_1+E_2$}
    \end{Overpic}
    \caption{This figure is obtained from Figure \ref{fig:R1} (they are \emph{not} Kirby diagrams). $E_1 + E_2$ and $E_2$ do not have intersection, and cutting $E_2$ would not change $E_1+E_2$.}
    \label{fig:slide_E1}
\end{figure}

After stretching $S_2$, the left-hand-side of Figure \ref{fig:single_E2_Kirby} becomes a punctured $S^2 \times S^2$. One can choose a metric such that any self-dual cohomology class is proportional to $[F]+ [E_1 + E_2]$. On the right-hand-side it is a punctured $\overline{\CP^2}$, and any self-dual cohomology class would be trivial. Take 
\[
\mathcal{B} = \langle F, E_1 +E_2, E_2\rangle
\]
as a basis of $H^2$, where we consider the Poinc\'are duals of the elements in $\mathcal{B}$. Under this basis any self-dual cohomology class is the element $(x,x,0)$. One can take $x = 1$ (We will prove it does not depend on the choice).

To stretch $S_1$,  we perform the following transformations to obtain a compatible Kirby diagram:
\begin{enumerate}
\item We start from Figure \ref{fig:V1_Kirby}.
\item Slide the handle representing $F$ over the handle representing $E_1$. The resulting handle represents $F+E_1$, and links only to the handle representing $E_2$. The resulting handle is not linked to $E_1$ since $E_1$ has self-intersection $-1$.
\item Slide the handle representing $E_2$ over the handle representing $E_1$. The resulting handle represents $E_1 +E_2$, and it is not linked to $E_1$.
\item Slide the handle representing $E_1 +E_2$ over the handle representing $F+E_1$. The resulting handle represents $F+ 2E_1 +E_2$ (see Figure \ref{fig:single_Kirby}), and they are unlinked from each other.
\end{enumerate}
\begin{figure}
    \begin{Overpic}{\includegraphics[scale=0.6]{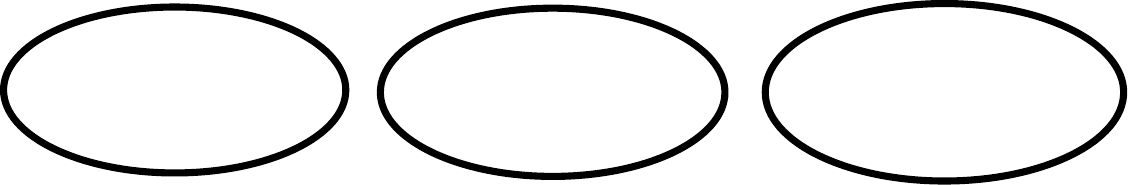}}
     \put(10,18){$1$}
     \put(0,-3){$F+E_1$}
    \put(50,-3){$F+ 2E_1+E_2$}
        \put(41,18){$-1$}
        \put(83,18){$-1$}
        \put(95,-3){$E_1$}
    \end{Overpic}
    \caption{}
    \label{fig:single_Kirby}
\end{figure}

On the resulting diagram, the intersection form is a diagonal matrix, and one can choose some metric such that the Hodge star operator is $(\id, -\id,-\id)$. Hence any self-dual cohomology class is proportional to $F+E_1$. Under the chosen basis 
\[
\mathcal{B} =\langle F, E_1 +E_2, E_2\rangle
\]
it looks like $(y,y,-y)$ for some $y$. One can take $y = 1$ (We will prove it does not depend on the choice). Reparameterize $t\in (-\infty, +\infty)$ by $t\in [0,1]$. Then during the homotopy $Q(R_1)$, the reference family $\kappa_{g_t}$ can be chosen to be $(1,1,t-1)$. We will prove it does not depend on the choice of the path.

In Figure \ref{fig:single_E2_Kirby}, we have a connected sum of $S^2\times S^2$ and $\overline{\CP^2}$. Since it has no 2-torsion in $H^2$, any spin$^c$-structure must have the form \[2aF + 2b(E_1+E_2)+(2c+1)E_2\]for some integer $a$, $b$, and $c$, which we write as
\[
(2a, 2b, 2c+1).
\]
The fixed perturbation $\omega$ is small enough in \eqref{equ:wall}, so without loss of generality, we have 
\begin{equation}\label{equ:cup-product-matrix}
( c_1(\fs) \cup [\kappa_{g_t}])[V_1]+  \frac{4i}{2\pi}\int_{V_1}\omega \wedge \kappa_{g_t}= (1,1,t-1)\begin{pmatrix}
0 & 1  & 0\\
1 & 0 & 0\\
0 &0 &-1
\end{pmatrix}\begin{pmatrix}
2a \\
2b\\
2c+1 
\end{pmatrix} + \epsilon_1 t+ \epsilon_2 =0
\end{equation}
which is
\begin{equation}\label{equ:wall-crossing-formula}
2a + 2b +(2c+1)(1-t)+ \epsilon_1 t+ \epsilon_2 = 0,
\end{equation}
for small fixed $\epsilon_1$ and $\epsilon_2$. 

Since the left-hand-side of \eqref{equ:wall-crossing-formula} is a linear function on $t$, it suffices to compare the evaluations at $t\in\{0,1\}$. There are two conditions for a spin$^c$-structure $(2a,2b,2c+1)$ to admit solutions:
\begin{enumerate}
\item $|2a + 2b| <|2c+1|$.
\item If $2a + 2b =0$ then $2c+1$ and $\epsilon_1+\epsilon_2$ do not have the same sign. If $2a + 2b \neq 0$ then $2c+1$ and $2a+2b$ do not have the same sign.
\end{enumerate}

We can describe the sign contribution of solutions from the wall-crossing as follows. When $t$ goes from $0$ to $1$, if for one solution, the left-hand-side of \eqref{equ:wall-crossing-formula} goes from negative to positive, and for another solution, it goes from positive to negative, then the sign contributions from the wall-crossing are opposite. If both are from negative to positive or positive to negative, then the sign contributions are the same.

We classify spin$^c$ structures that satisfy this equation for some $t\in [0,1]$ by following three types:
\begin{enumerate}
\item $2a + 2b \neq 0$.
\item $2a + 2b =0$ and $ a \neq 0$.
\item $a=b=0$.
\end{enumerate}
We will examine the cobordism maps induced by these structures and show that the first two types contribute to $0$ and the last type contributes to an isomorphism.

Recall that in the triple composite 
\[
V_1 = W_1\cup_{Y_2} W_2\cup_{Y_3}W_3,
\]
we have $6$ disks $\Dc_i$ and $\Dcc_i$ in $W_i$ for $i=1,2,3$ and the superscripts representing the core disk and the cocore disk, respectively

Note that $\partial \Dc_1 \subset S^3$, so the holonomy around $\partial \Dc_1$ is zero. So by Stokes' theorem,
\[
\int_{\Dc_1} F_{A^t_z} =0.
\]
Similarly we have 
\[
\int_{\Dcc_3} F_{A^t_z} =0.
\]
Hence we only need to consider the behavior of those structures on $E_1 = \Dcc_1\cup \Dc_2$ and $E_2 = \Dcc_2\cup \Dc_3$.

For type (1) structures, we have the following symmetry:
\[
(2a, 2b,2c+1) \leftrightarrow (-2a, -2b,-2c-1)
\]
Notice that if $c_1(\fs) = (2a, 2b,2c+1)$, then from the intersection form, we have
\begin{align*}
\frac{i}{2\pi}\int_{E_1} F_{A^t} &= \langle c_1(\fs), [E_1]\rangle\\
&= (0,1,-1)\begin{pmatrix}
0 & 1  & 0\\
1 & 0 & 0\\
0 &0 &-1
\end{pmatrix}\begin{pmatrix}
2a \\
2b\\
2c+1 
\end{pmatrix} \\
&= 2a+2c+1
\end{align*}
and 
\begin{align*}
\frac{i}{2\pi}\int_{E_2} F_{A^t} &= \langle c_1(\fs), [E_2]\rangle \\
&= (0,0,1)\begin{pmatrix}
0 & 1  & 0\\
1 & 0 & 0\\
0 &0 &-1
\end{pmatrix}\begin{pmatrix}
2a \\
2b\\
2c+1 
\end{pmatrix}\\
&= -2c-1
\end{align*}
So the map of local system is given by the multiplication with
\[
(u^{2a+2c+1}u^{-2c-1}u^\mathcal{T})^{\frac{1}{2}}
\]
for some topological term $\mathcal{T}$. 

Similarly, for $(-2a, -2b,-2c-1)$ the local system map is given by 
\[
(u^{-2a-2c-1}u^{2c+1}u^\mathcal{T})^{\frac{1}{2}}
\]
Note that for each pair $Dcc_i \cup Dc_{i+1}$ the Euler number is $2$ and the self-intersection is $-1$. So no matter how we choose the framing (for example, we adopt the setting in Proposition \ref{prop: first homotopy example}), we have $\mathcal{T} = 3\cdot (2-1) = 3$, though we do not need this fact because the topological terms are the same for all spin$^c$ structures. Then the difference between this pair is 
\[
(u^{2a+2c+1}u^{-2c-1}u^3)^{\frac{1}{2}}(u^{-2a-2c-1}u^{2c+1}u^3)^{-\frac{1}{2}}=u^{2a}.
\]
When $u = -1$, this term is $1$. Meanwhile, these solutions have the wall crossing over them with opposite sign contributions. Hence the overall effect is that $(2a, 2b,2c+1)$ cancels out with $ (-2a, -2b,-2c-1)$.

For type (2) structures, we conclude that they have no solutions from the following lemma \blem\label{lem: no solution}
For spin$^c$-structure $(2a,2b,2c+1)$ where $a \neq 0$, there is no solution. \elem
\begin{proof}   
Suppose $[S^2]$ is a cohomology class represented by $\{\pt\}\times S^2 \subset S^1\times S^2 = R_1$. We learn from Figure \ref{fig:R1} that $[S^2] = [E_1] + [E_2]$. Then from the intersection form, we have
\begin{align*}
\langle c_1(\fs), [S^2]\rangle
&=
\langle c_1(\fs), [E_1+E_2]\rangle\\
&=(0,1,0)\begin{pmatrix}
0 & 1  & 0\\
1 & 0 & 0\\
0 &0 &-1
\end{pmatrix}\begin{pmatrix}
2a \\
2b\\
2c+1 
\end{pmatrix} \\
&= 2a
\end{align*}
However, for $\langle c_1(\fs), [S^2]\rangle \neq 0$, there are no critical points of the Chern--Simons--Dirac functional. This is from the proof of \cite[Proposition 4.2.1]{kronheimer2007monopoles}, which is ultimately from the fact that reducible solutions on closed $3$-manifolds must be in torsion spin$^c$ structures.
\end{proof}

\begin{remark}
    Lemma \ref{lem: no solution} applies to both type (1) and (2) structures, which is important to show that there are only finitely many solutions in the collection of spin$^c$ structures with fixed absolute value of the $\langle c_1(\fs),[E_1]\rangle$, where this fixed condition corresponds to the fixed power of $U$ from the discussion in \cite[Page 513]{kronheimer2007monopolesandlens}. On the other hand, the discussion of type (2) structures in \cite[Page 91]{Freeman2021triangle} does not work for the following reasons. Freeman used the partial symmetry \[(2a,2b,2c+1)\leftrightarrow(-2a,-2b,2c+1)\]to cancel the solutions, where the symmetry is from the spin$^c$ conjugation on the summand $S^2\times S^2$ in the decomposition of Figure \ref{fig:single_E2_Kirby}. However, the computation of the signs in his local system indeed uses the partial symmetry from the spin$^c$ conjugation on the summand $\ov{\CP^2}$ containing $E_1$ in Figure \ref{fig:single_Kirby}, which induces\[(2a,2b,2c+1)\leftrightarrow(-2a-2(2c+1),2(2a+b+2c+1),2c+1)\]via the basis change\[\begin{aligned}
        2aF+&2b(F+E_1)+(2c+1)E_2\\&=(2a+2b-2c-1)(F+E_1)+(2c+1)(F+2E_1+E_2)-(2a+2c+1)E_1.
    \end{aligned}\]Meanwhile, Freeman does not write the perturbation term in \eqref{equ:cup-product-matrix} explicitly, which makes the discussion on type (2) and (3) structures become subtle because the solution is at the endpoint $t=1$.
\end{remark}
Finally, we come to the only nonzero contribution, which is from the type (3) spin$^c$-structures. From the second condition for \eqref{equ:wall-crossing-formula} to admit a zero, we know that only one structure of the symmetry pair
\[
(0, 0,2c+1) \leftrightarrow (0, 0,-2c-1)
\]
contributes a solution. From the dimension formula \cite[Lemma 5.7, Corollaries 5.8, and 5.9]{kronheimer2007monopolesandlens}, we deduce that this structure contributes a factor of $\pm U^{k(k+1)/2}$. So the overall contribution of $L$ is 
\[
\sum_{k\geq 0}\pm U^{k(k+1)/2}.
\]
In particular, the leading term $\pm 1$ is invertible. So $L$ is an isomorphism and we finish the proof of Proposition \ref{prop:L-iso}.

\section{Triangle over integer coefficients}\label{sec: Triangle over integer coefficients}
In this section, we prove Theorem \ref{thm: main triangle} by picking suitable $\omega$ in Theorem \ref{thm: local system triangle} and show that Theorem \ref{thm: main triangle} recovers the triangle in \cite[\S 4]{LRS2023triangle} over $\Q$.

\subsection{Homological trick}
In this subsection, we apply the homological trick to obtain the surgery exact triangle without local system from that with local system. The main idea follows from \cite[\S 3]{LY2025dimension}. See also \cite[\S 2.3]{alfieri2020framed} and \cite[\S 2.2]{baldwin2020concordance}. We start with the following notation and computation. 


Let $\Dc_i$ and $\Dcc_i$ denote the core disk and the cocore disk of the surgery cobordism $W_i:Y_i\to Y_{i+1}$ for $i\in\Z/3$ and let $D_i\subset Y_i$ be the meridian disk of the Dehn filling solid torus, oriented by its boundary $\ga_i$. Note that the condition on $\ga_i$ implies that \[\ga_0+\ga_1+\ga_2=0\aand K_i=\ga_{i+1}\]in the boundary torus. Note that \[\partial \Dc_i=-K_i=-\ga_{i+1},~\partial \Dcc_i=K_{i+1}=\ga_{i+2}=\ga_{i-1},\]\[\aand\partial (\ga_i\times I)=(-\ga_i\times \{0\})\cup (\ga_i\times \{1\}),\]where the orientations of the core and cocore disks are from their boundaries. Up to isotopy, we have\[\Dc_i=(\ga_{i+1}\times I)\cup (-D_{i+1})\aand \Dcc_i=(-D_i)\cup (-\ga_i\times I)\cup (-D_{i+1}),\]for which we can compute the boundary as a sanity check. 


Let $\nu_i\subset W_i$ be the union of $\Dc_i$ and $\Dcc_i$. For fixed $j\in\Z/3$, we pick $\omega=-\ga_j$, isotoped to be away from the surgery region. Since the indices are cyclically ordered, the choice of $j$ does not matter. To be consistent with the notation in \cite[\S 4]{LRS2023triangle}, we pick $j=1$. Let $\nu_i'=\nu_i\cup (\omega \times I)$. Then \[\begin{aligned}
    \nu_2'=&(\ga_0\times I)\cup (-D_0)\cup (-D_2)\cup (-\ga_2\times I)\cup (-D_0)\cup (-\ga_1\times I)\\=&2\cdot(\ga_0\times I)\cup 2\cdot (-D_0)\cup (-D_2)\\=&2\cdot \Dc_2\cup (-D_2)
\end{aligned}\]Since\[\Dc_0=(\ga_1\times I)\cup (-D_1)\aand \Dcc_1=(-D_1)\cup (-\ga_1\times I)\cup (-D_2),\]we have\[\nu_0'= (-D_1)\cup \Dcc_0\aand \nu_1'=D_1\cup\Dc_1\cup 2\cdot \Dcc_1\cup D_2.\]

Let\[\nu_2''=\nu_2'\cup D_2=2\cdot \Dc_2,~\nu_0''=\nu_0'\cup D_1=\Dcc_0,\]\[\aand \nu_1''=(-D_1)\cup \nu'_1\cup (-D_2)=\Dc_1\cup 2\cdot \Dcc_1.\]

Then we have the following corollary of Theorem \ref{thm: local system triangle}.

\bcor\label{cor: triangle 1}
Let $(Y_0,Y_1,Y_2)$ be a surgery triad and let $W_i:Y_i\to Y_{i+1}$ for $i\in\Z/3$ be the corresponding surgery cobordism. Let $\Dc_i$ and $\Dcc_i$ be the core disk and the cocore disk in the corresponding cobordism, respectively. Then there exists an exact triangle
\[	\xymatrix{
	HM(Y_0)\ar[rr]^{HM(W_0;\Dcc_0)}&& HM(Y_1;K_1)\ar[dl]^{\quad HM(W_1;\Dc_1)}\\
	&HM(Y_2)\ar[lu]^{HM(W_2)\quad}&
	}\]
\ecor
\bpf
From Theorem \ref{thm: local system triangle} and the above computation, we obtain the following exact triangle\[	\xymatrix{ HM(Y_0)\ar[r]^<<<<<<<{HM(W_0;\nu_0'')}& HM(Y_1;K_1)\cong HM(Y_1;K_1\cup \ga_1)\ar[d]^{\quad HM(W_1;\nu_1'')}\\
	&HM(Y_2)\cong HM(Y_2;\ga_2)\cong HM(Y_2;2 K_2)\ar[lu]^{HM(W_2;\nu_2'')\quad}
	}\]where the isomorphisms are induced by those for $D_1$ and $D_2$ from Lemma \ref{lem: I_S is an iso}, and those for $2\cdot\Dc_2$ and $2\cdot \Dcc_1$ from Lemma \ref{lem: cob map depending on H_2}.
\epf
\brem
One may use Proposition \ref{prop: first homotopy example} and a modification of the local system set in the construction of the second homotopy in \S \ref{subsec: second homotopy} to prove Corollary \ref{cor: triangle 1} directly.
\erem
We want to further remove the local system set $K_1$ in Corollary \ref{cor: triangle 1} to obtain Theorem \ref{thm: main triangle}. This is achieved by assuming $[K_1]=0\in H_1(Y_1;\Z/2)$, or equivalently $[\ga_2]\in H_1(Y_1;\Z/2)$. This is always possible for one index $i$ because\[\ker\big(H_1(\partial (Y_i\backslash \mathring{\nu}(K_i));\Z/2)\to H_1( Y_i\backslash \mathring{\nu}(K_i);\Z/2)\big)\cong \Z/2\]from the ``half-live-half-die" theorem and $\ga_0,\ga_1,\ga_2$ represent the three nontrivial elements in \[H_1(\partial (Y_i\backslash \mathring{\nu}(K_i));\Z/2)\cong (\Z/2)^2.\]
\bcor[{Theorem \ref{thm: main triangle}}]\label{cor: triangle 2}
Let $(Y_0,Y_1,Y_2)$ be a surgery triad and let $W_i:Y_i\to Y_{i+1}$ for $i\in\Z/3$ be the corresponding surgery cobordism. Let $\Dc_i$ and $\Dcc_i$ be the core disk and the cocore disk in the corresponding cobordism, respectively. Suppose $[K_1]=0\in H_1(Y_1;\Z/2)$ and suppose $S\subset Y_1$ is a $2$-chain with $\partial S=K_1-2\eta$ for some $1$-chain $\eta\subset Y_1$. Note that such $S$ always exists. Then there exists an exact triangle
\[	\xymatrix{
	HM(Y_0)\ar[rr]^<<<<<<<<<<<<<<<{HM(W_0;\Dcc_0\cup (-S))}&& HM(Y_1)\cong HM(Y_1;2\eta)\ar[dl]^{\quad HM(W_1;\Dc_1\cup S)}\\
	&HM(Y_2)\ar[lu]^{HM(W_2)\quad}&
	}\]
\ecor
\bpf
This follows from Corollary \ref{cor: triangle 1} and Lemma \ref{lem: I_S is an iso}.
\epf
\subsection{Recovering triangle over rationals}
In this subsection, we show that the exact triangle in Corollary \ref{cor: triangle 2} recovers the triangle in \cite[Theorem 4.6]{LRS2023triangle} over $\Q$. The strategy in the reference is to provide an ad hoc arrangement of signs for cobordism maps in different spin$^c$ structure components so that the composition of two consecutive modified cobordism maps in the triangle vanishes. We will show that this ad hoc arrangement coincides with our signs from the local system, so that it becomes more natural.

We review the ad hoc sign arrangement as follows. We write $M=Y_i\backslash  \mathring{\nu}(K_i)$ for the knot complement, which is denoted by $Y$ in the reference. We assume that \[[K_1]=[\ga_2]=0\in H_1(M;\Z/2).\]Let $\fs_0$ be a fixed self-conjugate spin$^c$ structure on $M$ and let $\spinc(W_i,\fs_0)$ denote the set of spin$^c$ structures on $W_i$ with restriction $\fs_0$ on $M$, which is an affine space over\begin{equation}\label{eq: affine}
    \ker(H^2(W_i)\to H^2(M)).
\end{equation}For $i\in\Z/3$, we fix a base element $\ft_i\in \spinc(W_i,\fs_0)$ (denoted by $\fs_1$ and $\fs_2$ in the reference) so that elements in $\spinc(W_i,\fs_0)$ are written as $\ft_i+h$ for $h$ in \eqref{eq: affine}. 


Define
\[
\mu : \bigcup_{i \in \Z/3} \spinc(W_i,\fs_0) \to \Z/2
\]
as follows.
\begin{itemize}
    \item $\mu$ is identically zero on $\spinc(W_2,\fs_0)$;
    \item For $i=0,1$ and $\fs=\ft_i+h\in \spinc(W_i,\fs_0)$, let $\mu(\fs) = h_{\Z/2}$, where
    \[
    h_{\Z/2} \in \ker(H^2(W_i; \Z/2) \to H^2(M; \Z/2)) = \Z/2
    \]
    is the mod $2$ reduction of $h$.
\end{itemize}
Then the modified cobordism map is\begin{equation}\label{eq: modified cob}
    F_{W_i}=\sum_{\substack{\fs_0\in \spinc(M)\\\text{self-conjugate}}}\sum_{\fs\in \spinc(W_i,\fs_0)}(-1)^{\mu(\fs)}\cdot HM(W_i,\fs).
\end{equation}
\bprop\label{prop: recover LRS}
For suitable choices of $\ft_0$ and $\ft_1$ for each self-conjugate $\fs_0\in \spinc(M)$, the modified cobordism map \eqref{eq: modified cob} agrees with the cobordism map in Corollary \ref{cor: triangle 2}.
\eprop
\bpf
The case $i=2$ is trivial. The proofs for $i=0$ and $i=1$ are similar. We only consider the case $i=0$. 

We first show that $\Dc_0$ vanishes under the composition \begin{equation*}
    H_2(W_0,\partial W_0;\Z/2)\xrightarrow{\cong}H^2(W_0; \Z/2) \to H^2(M; \Z/2).
\end{equation*}For any element $\sigma\in H_2(M;\Z/2)$, the evaluation of the image of $\Dc_0$ is the intersection of $\Dc_0$ with the image of $\sigma$ in $H_2(W_0;\Z/2)$. Since $\Dc_0$ is disjoint from $M$, we know there is no intersection point and hence we obtain the desired fact.

Suppose $S\subset Y_1$ is a $2$-chain with $\partial S=K_1-2\eta$ for a $1$-chain $\eta\subset M$. Let $\Sigma=\Dcc_0\cup (-S)$ and let $\Sigma'$ be the union of $\Sigma$ and an unoriented annulus $\eta\times I$, which is a non-orientable closed $2$-chain and represents a homology class in $H_2(W_0;\Z/2)$.

Since $\Sigma'$ and $\Dc_0$ intersect at one point, we know $\Sigma'$ is a generator of $H_2(W_0;\Z/2)$ and $\Dc_0$ is a generator of $H_2(W_0,\partial W_0;\Z/2)$. Hence the Poinc\'are dual of $\Dc_0$ generates \begin{equation*}
 \ker(H^2(W_0; \Z/2) \to H^2(M; \Z/2))=\Z/2.
\end{equation*}
Then we have 
\[h_{\Z/2}=\langle h, \Sigma'\rangle_{\Z/2},\]where $\langle \cdot,\cdot\rangle _{\Z/2}$ is the pairing for homology and cohomology over $\Z/2$. 

We first deal with the special case $\eta=0$. Then $\Sigma'=\Sigma$ and \[h_{\Z/2}=\langle h, \Sigma\rangle\pmod 2,\]

On the local system side, from \eqref{eq: cob map} and Example \ref{exmp: closed}, we know \[HM(W_0;\Sigma,\fs)=(-1)^{\frac{1}{2}(\langle c_1(\fs),\Sigma\rangle+\chi(\Sigma)+\Sigma\cdot \Sigma)}\cdot HM(W_0,\fs).\]Note that $\chi(\Sigma)+\Sigma\cdot \Sigma$ is independent of $\fs$, and \[\frac{\langle c_1(\ft_0+h),\Sigma\rangle}{2}-\langle h,\Sigma\rangle=\frac{\langle c_1(\ft_0),\Sigma\rangle}{2}\]is also independent of $\fs$. Hence we have\[HM(W_0;\Sigma)=(-1)^{\frac{1}{2}(\langle c_1(\ft_0),\Sigma\rangle+\chi(\Sigma)+\Sigma\cdot \Sigma)}\cdot F_{W_0}.\]We can choose $\ft_0$ for $\fs_0$ so that $\langle c_1(\ft_0),\Sigma\rangle+\chi(\Sigma)+\Sigma\cdot \Sigma$ is a multiple of $4$. Hence the two maps agree.

Now we prove the case for $\eta\neq 0$. Since $h$ is in \eqref{eq: affine}, it has a lift $\wti{h}\in H^2(W_0,M)$. Note that $\Sigma$ represents a homology class in $H_2(W_0,M)$. Then we also have\begin{equation}\label{eq: h2 new def}
    h_{\Z/2}=\langle \wti{h},\Sigma\rangle \pmod 2.
\end{equation}For different choices of lifts $\wti{h}$ and $\wti{h}'$, we have\[\langle \wti{h}-\wti{h}',\Sigma\rangle=\langle \delta H,\Sigma\rangle =\langle H,\partial \Sigma\rangle=2\langle H,\eta\rangle\]where $H\in H^1(M)$ and $\delta:H^1(M)\to H^2(W_0,M)$ is the boundary map. Hence \eqref{eq: h2 new def} does not depend on the lift. To pin down the evaluation $\langle \wti{h},\Sigma\rangle$ over $\Z$, we need to fix a lift $\ov{h}\in H^2(W_0,\partial \Sigma)$, or equivalently a framing $\tau_h$ of $L_h|_{\partial \Sigma}$, the restriction of the complex line bundle $L_h$ with $c_1(L_h)=h$. We fix framings $\tau_h$ for all $h\in H^2(W)$ that are the same for two copies of $\eta$, and hence fix the lifts $\ov{h}$. Since\[P_{\det}(\ft_0+h)=P_{\det}(\ft_0)\ot L_h^{\ot 2},\]the framing of $P_{\det}\ft_0|_{\partial \Sigma}$ and $\tau_h$ induce a framing of $P_{\det}(\ft_0+h)|_{\partial \Sigma}$.

We pick a framing \[\tau(\eta)=(\tau^\fs(\eta))_{\fs\in \spinc(Y_1)}=(\tau^\fs_{\det}(\eta),\tau^\fs_N(\eta),\tau^\fs_T(\eta))_{\fs\in \spinc(Y_1)}\]of $\eta$ as in Definition \ref{defn: framing} such that $\tau_N(\eta)$ and $\tau_T(\eta)$ are induced by the boundary framing of $\Sigma$, and $\tau^\fs_{\det}(\eta)$ is induced by the above fixed choice of $\tau_h$ and $\tau^{\ft_0}_{\det}$. Note that $\tau(2\eta)$ is always admissible by Remark \ref{rem: admissible 2eta}. We consider the relative Euler class \[e(P_{\det}\fs|_{\Sigma},\tau^\fs_{\det}(2\eta))\in H^2(\Sigma,\partial \Sigma=2\eta)\]and use $\fs$ to denote $P_{\det}\fs|_{\Sigma}$ for simplicity. Then we have \[\frac{\langle e(\ft_0+h,\tau^{\ft_0+h}_{\det}(2\eta)),\Sigma\rangle}{2}-\langle \ov{h},\Sigma\rangle=\frac{\langle e(\ft_0,\tau^{\ft_0}_{\det}(2\eta)),\Sigma\rangle}{2}.\]

From \eqref{eq: cob map 2} and the proof of Lemma \ref{lem: even}, we have a generalization of Example \ref{exmp: closed}\[HM(W_0;\wti{\Sigma},\fs)=(-1)^{\frac{1}{2}(\langle e(\fs,\tau^\fs_{\det}(\eta)),\Sigma\rangle+\chi(\Sigma)+\Sigma\cdot \Sigma)}\cdot HM(W_0,\fs),\]where $\wti{\Sigma}=(\Sigma',\kappa=\eta\times \{1/2\},\tau(\kappa))$ is the enhanced local system set from Remark \ref{rem: enhanced level set}. The rest of the proof is similar to the case of $\eta=0$.
\epf
\section{Spectral sequence}\label{sec: Spectral sequence}
In this section, we prove Theorems \ref{thm: main ss} and \ref{thm: main kh} by adapting the results in \cite{Bloom2011link,scaduto2015instantons,kronheimer2011khovanov} and applying the trick in \S \ref{sec: Triangle over integer coefficients}. We will omit verbatim details in the proof and only point out the difference. Throughout this section, let $|v|$ for $v\in\{0,1\}^l$ be the $L^1$ norm of $v$ and equip $\{0,1\}^l$ with a partial order induced by $0<1$.

First, as a standard iterating generalization of Theorem \ref{thm: local system triangle} for a link, we obtain the following spectral sequence.
\bprop\label{prop: ss local system}
Let $L$ be a framed link in a closed oriented $3$-manifold $Y$ with ordered $l$ components. Let $v\in\{0,1\}^l$ and let $Y_v=Y_v(L)$ be obtained from $Y$ by surgery on $L$ via the slopes in $v$. Let $\eta_{v}$ be the union of core knots in $Y_v$. Let $\omega\subset Y\backslash L$ be a (possibly empty) $1$-submanifold away from the surgery region. Then there exists a spectral sequence depending only on $(Y,L)$\[E^1=\bigoplus_{v\in \{0,1\}^l}HM(Y_v;\eta_v\cup \omega)	\Longrightarrow HM(Y;L\cup \omega).\]Moreover, we have\[d^1=\sum_{\substack{|v-w|=1\\v<w}} (-1)^{\delta(v,w)}HM(W_{vw},\mu_{vw};\nu_{vw}\cup (\omega\times I)),\]where $W_{vw}=W_{vw}(L)$ is the surgery cobordism with source $Y_v$ and target $Y_w$, $\nu_{vw}$ is the union of the core disks and the cocore disks in $W_{vw}$, $\delta(v,w)=\sum_{i=1}^jv_i$ for the entry $j$ that $v$ and $w$ differ by, and $\mu_{vw}$ is the homology orientation from \cite[Lemma 6.1]{kronheimer2011khovanov} satisfying $\mu_{vw}=\mu_{uw}\circ \mu_{vu}$ for $v\le u\le w$ in $\{0,1\}^l$. Moreover, the spectral sequence inherits a $\Z\op\Z/2$-grading so that the grading shift of the differential $d^r$ in the $r$-th page is $(r,1)$.
\eprop
\bpf
We first consider the case of $HM=\widecheck{HM}_\bu$. The cases of $\widehat{HM}_\bu$ and $\overline{HM}_\bu$ are similar. Later, we will consider the case of $\widetilde{HM}_\bu$.

The proof is an adaptation of \cite[\S 4 and 7]{Bloom2011link} with signs from \cite[\S 6]{scaduto2015instantons} and \cite[\S 7]{kronheimer2011khovanov}. See also the end of \cite[\S 2]{Lin19involutive} for a similar setup without signs. We mostly follow \cite[\S 6]{scaduto2015instantons} and use the similar notation. 

For $v\le w$ in $\{0,1\}^l$, let $W_{vw}$ be the composition of surgery cobordisms between $Y_v$ and $Y_w$ and let $\nu_{vw}\subset W_{vw}$ be the union of the corresponding local system sets. In particular, we set $(W_{vv},\nu_{vv})=(Y_v,\nu_v)\times I$. 


Let $G_{vw}$ be the family of metrics on $W_{vw}$ constructed in \cite[\S 6.1]{scaduto2015instantons}. Let\[C(Y;\eta)=\widecheck{CM}_\bu(Y;\eta)\aand\]\[ m_{vw}=\widecheck{CM}_\bu(W_{vw},G_{vw},\mu_{vw};\nu_{vw}\cup (\omega\times I)):C(Y_v;\eta_v\cup \omega)\to C(Y_w;\eta_w\cup \omega),\]where $\widecheck{CM}_\bu(Y,\eta)$ is the chain complex of $\widecheck{HM}_\bu(Y,\eta)$, $m_{vw}$ is the map associated to the family of metrics constructed in the proof of \cite[Proposition 25.3.8]{kronheimer2007monopoles}. Here we orient the moduli spaces as in \cite[Equation (4.5)]{scaduto2015instantons}. See also the discussion after \cite[Equation (20)]{kronheimer2011khovanov}, which mentions a correction of signs in the proof of \cite[Proposition 25.3.8]{kronheimer2007monopoles}.

Define \[ s(v,w) = \frac{(|w - v|^2 - |w - v|)}{2} + |v|,
\]\[
C = \bigoplus_{v \in \{0,1\}^l} C(Y_v;\eta_v),\aand \pa = \sum_{v \leq w} \partial_{vw}=\sum_{v \leq w} (-1)^{s(v,w)}m_{vw}.
\]

Then the $C(Y_v;\eta_v) \to C(Y_w;\eta_{w})$ component of $\partial^2$ is

\[(-1)^{s(v,w)+|w|} \sum_{v \leq u \leq w} (-1)^{|w-u|(|u-v|-1)} m_{uw} m_{vu},\]which vanishes by \cite[Equation (6.1)]{scaduto2015instantons}.

Define the cubical filtration $FC$ on $C$ by \[F^iC=\bigoplus_{|v|\ge i}C(Y_v;\eta_v),\]which induces a spectral sequence with the $E^1$-page as in the statement of the proposition.

Similar to the proof of \cite[\S 6.3]{scaduto2015instantons}, we obtain that $(C,\partial)$ is quasi-isomorphic to $\widecheck{CM}_\bu(Y;L\cup \omega)$ via the composition of $l-1$ quasi-isomorphisms collapsing the cube in one coordinate. The ingredients of the construction of the quasi-isomorphisms are again the triangle detection lemma (cf.\ \cite[Lemma 5.1]{scaduto2015instantons} and \cite[Lemma 7.1]{kronheimer2011khovanov}) and the proof of the surgery exact triangle in Theorem \ref{thm: local system triangle}.

The $\Z/2$-grading and the invariance of the spectral sequence for auxiliary choices follow from \cite[\S 7]{Bloom2011link}. 

The case of $\widetilde{HM}_\bu$ follows from \cite[\S 8]{Bloom2011link}. For more detail, the $U$-map is defined by removing a $4$-ball in the cobordism and inserting an element in the perturbed monopole moduli space of $S^3$ corresponding to $U$. The chain complex $\widetilde{CM}_\bu$ of $\widetilde{HM}_\bu$ is defined as the mapping cone of such $U$-map. Based on this cobordism interpretation of $U$, the spectral sequence is obtained by extending the complex $(C,\partial)$ for $\widecheck{HM}_\bu$ to a $(l+1)$-dimensional cube with the last coordinate corresponding to $U$.
\epf
Then we prove the main theorems.

\bpf[Proof of Theorem \ref{thm: main ss}]
We generalize Corollary \ref{cor: triangle 1} to the case of a link. In the knot case, to make the local system set on $W_{0}$ vanish, we need to pick $\omega=\ga_2$, which is the meridian of the knot in $Y_2=Y$. In the general case, we pick \[\omega=\bigcup_{i=1}^lm_i\]for the meridian $m_i$ of the $i$-th component of $L$. Then the proof of Corollary \ref{cor: triangle 1} applies to each component of the link and we obtain the desired result from Proposition \ref{prop: ss local system}.
\epf
\brem
To generalize Corollary \ref{cor: triangle 2} to the case of a link and remove the local system sets for all $3$-manifolds involved in the cube and collapsed cubes in the proof of Theorem \ref{thm: main ss}, we need to assume that each component $K$ of $L$ satisfies $[K]=0\in H_1(Y\backslash (L\backslash K);\Z/2)$. This might not be satisfied in general (even for the core link $\eta_v\subset Y_v$) because now the ``half-live-half-die" theorem involves multiple boundary components of $Y\backslash\mathring{\nu}(L)$.
\erem
\bpf[Proof of Theorem \ref{thm: main kh}]
This follows from Theorem \ref{thm: main ss}, \cite[Proposition 9.1]{Bloom2011link}, \cite[Proposition 36.1.3]{kronheimer2007monopoles}, \cite[\S 7.8 and \S 8]{scaduto2015instantons}. Note that the surgery link $L'\subset \Sigma_2(L)$ represents the trivial homology class in $H_1(\Sigma_2(L);\Z/2)$ by the discussion after \cite[Equation (2.1)]{scaduto2015instantons}. Hence we can remove it from the local system set by Lemma \ref{lem: I_S is an iso}. The computation of the $\Z/2$-grading is from \cite[Theorem 1.5]{Bloom2011link}.
\epf
\section{Further direction on sutured monopole theory}\label{sec: Further direction on sutured monopole theory}
In this section, we sketch a modified definition of sutured monopole homology $SHM(M,\ga)$ of a balanced sutured manifold $(M,\ga)$ with our local system in \S \ref{sec: Local system}. We expect that the readers are familiar with \cite{kronheimer2007monopoles,kronheimer2010knots} and we use the notation in the reference freely.

To define $SHM(M,\ga)$, we need to pick a closure $(Y,R)$ of $(M,\ga)$, take\[SHM(M,\ga)=HM(Y|R)=\bigoplus_{\substack{\fs\in\spinc(Y)\\\langle c_1(\fs),R\rangle=2g(R)-2}}\widecheck{HM}_\bu(Y,\fs),\]and then prove the isomorphism class is independent of the choice of the closure. Note that $g(R)$ needs to be greater than $1$ so that all involved spin$^c$ structures are nontorsion. In the proof of the independence of $g(R)$ in \cite[Theorem 4.4]{kronheimer2010knots}, monopole Floer homology with the local system from \cite[\S 22.6 and \S 23.3]{kronheimer2007monopoles} was used. 

We hope to see what happens if we replace Kronheimer--Mrowka's local system with ours. The advantage of our local system is that the surgery exact triangle holds when $u=-1$, while one needs $2=0$ for the triangle over Kronheimer--Mrowka's local system (see \cite[Theorem 45.4.1]{kronheimer2007monopoles} and \cite[Theorem 5.12]{kronheimer2007monopolesandlens}). Following the notation in \S \ref{sec: Local system}, we write $\Ga'_\eta$ for Kronheimer--Mrowka's local system and $\Ga_\eta$ for our local system. In \cite{kronheimer2010knots}, Kronheimer--Mrowka further assumed that $t-t^{-1}$ is invertible for $t=\exp(1)$. We adopt this assumption in this section. Let the fiber of $\Ga_\eta$ be $\Z$ by setting $u=-1$. From Remark \ref{rem: u=t^-2}, we roughly have $u=t^2$ because there is a factor $1/2$. Hence $t-t^{-1}$ being invertible corresponds to $2$ being invertible in our construction, where $i$ disappears because we use the topological term to normalize the power.

The first important property is that if there is an integer cohomology class that evaluates at $1$ on $[\eta]$, then \begin{equation}\label{eq: HM zero}
    \ov{HM}_\bu(Y;\Ga'_{\eta})=0\aand \widecheck{HM}_\bu(Y;\Ga'_{\eta})\cong \widehat{HM}_\bu(Y;\Ga'_{\eta}),
\end{equation}which is \cite[Proposition 2.1]{kronheimer2010knots} and ultimately from \cite[Proposition 32.3.1]{kronheimer2007monopoles}. In such a case, we will use $HM$ to denote either $\widecheck{HM}_\bu$ or $\widehat{HM}_\bu$. The proof of \cite[Proposition 32.3.1]{kronheimer2007monopoles} starts with the general local system and then applies to $\Ga'_{\eta}$. Hence we can also apply the proof to $\Ga_\eta$ and obtain the following.

\bprop\label{prop: vanish}
If there is an integer cohomology class that evaluates at $1$ on $[\eta]$, then $\ov{HM}_\bu(Y;\Ga_{\eta})$ is annihilated by some power of $2$. If we do not set $u=-1$ in the local system, then $2$ is replaced by $1-u$.
\eprop
\bpf
The last paragraph in the proof of \cite[Proposition 32.3.1]{kronheimer2007monopoles} now implies that some power of $1-u^{p}$ annihilates $\ov{HM}_\bu(Y;\Ga_{\eta})$ because we take a factor of $1/2$ (see Remark \ref{rem: u=t^-2}) and the topological term vanishes. Note that $u=-1$ by our setting and $p=1$ from the evaluation assumption.
\epf
From Proposition \ref{prop: vanish}, if we consider the local system\[\widetilde{\Ga}_{\eta}=\Ga_{\eta}\ot \sR\]for some ring $\sR$ in which $2$ is invertible (e.g.\ $\sR=\Z[1/2], \Q$), then we have the similar result as in \eqref{eq: HM zero}. An interesting coincidence is that this invertible condition also appears in sutured instanton homology for which we need to take generalized eigenspaces of an operator with eigenvalues being multiple of $2$.

The second important property is that for $\eta=S^1\times \pt\times \pt \subset T^3$, we have\[HM(T^3;\Ga'_\eta)\cong \cR.\] From \cite[\S 37.3]{kronheimer2007monopoles}. Moreover, let $\widehat{E(1)}$ be the complement of the neighborhood of a regular fiber in a rational elliptic surface $E(1)$ and let $\widehat{\nu}\subset \widehat{E(1)}$ be obtained from a section $\nu\subset E(1)$ meeting the neighborhood of the fiber transversely in a disk. Then there is a relation between the relative invariants\[\psi(D^2\times T^2,D^2\times \pt)=(t-t^{-1})^{-1}\cdot \psi(\widehat{E(1)},\widehat{\nu})\in HM(T^3;\Ga'_\eta)\]from the end of \cite[\S 38.2]{kronheimer2007monopoles}.

We have the following corresponding result for our local system.

\bprop\label{prop: computation}
In the above setting, we have $HM(T^3,\widetilde{\Ga}_\eta)\cong \sR$ and \[\psi(D^2\times T^2,D^2\times \pt)=\pm \frac{1}{2}\cdot \psi(\widehat{E(1)},\widehat{\nu})\in HM(T^3,\widetilde{\Ga}_\eta)\]
\eprop
\bpf
The proofs in \cite[\S 37.3 and \S 38.2]{kronheimer2007monopoles} apply verbatim, except that we set $u=t^2=-1$. We need to tensor with $\sR$ so that Proposition \ref{prop: vanish} allows us to conclude $\ov{HM}_\bu(T^3,\widetilde{\Ga}_\eta)=0$. Note that the topological term must be odd and provide a factor of $t$. The ambiguity of the sign is from the topological term, which can be further calculated by considering the topological property of $\widehat{\nu}$.  
\epf

With Propositions \ref{prop: vanish} and \ref{prop: computation}, we may follow the definition of sutured instanton homology in \cite[\S 7]{kronheimer2010knots} to construct a modified sutured monopole homology as a $\sR$-module. More precisely, take a twisted closure $(Y,R,\omega)$ as in \cite[\S 7.4]{kronheimer2010knots} and define\[SHM(M,\ga;\widetilde{\Ga}_\omega)=HM(Y;\widetilde{\Ga}_\omega|R)\bigoplus_{\substack{\fs\in\spinc(Y)\\\langle c_1(\fs),R\rangle=2g(R)-2}}\widecheck{HM}_\bu(Y,\fs;\widetilde{\Ga}_\omega).\] Then one might show that the isomorphism class of $SHM(M,\ga;\widetilde{\Ga}_\omega)$ is independent of the choice of $(Y,R,\omega)$ similar to the proof for sutured instanton homology. Moreover, one might construct a projective transitive system of $SHM(M,\ga;\widetilde{\Ga}_\omega)$ and consider contact elements following Baldwin--Sivek \cite{BS2015naturality,baldwin2016contact}. Similar to \cite[Lemma 4.9]{kronheimer2010knots}, we might have\[SHM(M,\ga;\widetilde{\Ga}_\omega)\cong SHM(M,\ga)\ot \sR\]if $\sR$ has no $\Z$-torsion. We leave the details to future work.

\bibliographystyle{alpha}

\end{document}